\documentclass[12pt,twoside]{amsart}
\usepackage{amssymb,amsmath,amsthm, amscd, enumerate, mathrsfs}
\usepackage{graphicx, hhline}
\usepackage[all]{xy}
\usepackage[usenames]{color}
\usepackage[dvipdfmx]{hyperref}

\title{On quasi-Albanese maps}
\author{Osamu Fujino} 
\date{2024/3/21, version 0.29}
\keywords{quasi-Albanese maps, quasi-Albanese 
varieties, quasi-abelian varieties, 
logarithmic Kodaira dimension, mixed Hodge structures, 
commutative complex Lie groups, semipositivity theorems, 
weak positivity}
\subjclass[2020]{Primary 14E05; Secondary 14E30, 14L40, 32M05, 14K99}
\address{Department of Mathematics, Graduate School of Science, 
Kyoto University, Kyoto 606-8502, Japan}
\email{fujino@math.kyoto-u.ac.jp}

\newcommand{\codim}[0]{{\operatorname{codim}}}

\newcommand{\Supp}[0]{{\operatorname{Supp}}}
\newcommand{\Exc}[0]{{\operatorname{Exc}}}
\newtheorem{thm}{Theorem}[section]
\newtheorem{lem}[thm]{Lemma}
\newtheorem{cor}[thm]{Corollary}

\newtheorem{conj}[thm]{Conjecture}
\newtheorem*{claim}{Claim}

\theoremstyle{definition}
\newtheorem{ex}[thm]{Example}
\newtheorem{step}{Step}
\newtheorem{defn}[thm]{Definition}
\newtheorem{rem}[thm]{Remark}
\newtheorem*{ack}{Acknowledgments} 
\newtheorem{say}[thm]{}
\makeatletter

\@addtoreset{equation}{section}
\makeatother
\begin{document}

\maketitle 

\begin{abstract}
We discuss Iitaka's theory of 
quasi-Albanese maps in details. 
We also give a detailed proof of Kawamata's theorem on 
the quasi-Albanese maps for varieties of the logarithmic Kodaira dimension 
zero. Note that Iitaka's theory is an application of 
Deligne's mixed Hodge theory for smooth 
algebraic varieties. 
\end{abstract}

\tableofcontents 

\section{Introduction}\label{p-sec1}

In this paper, we discuss Iitaka's theory of quasi-Albanese maps. 
We give a detailed proof of: 

\begin{thm}[{see \cite{iitaka1} and Theorem \ref{p-thm3.16}}]\label{p-thm1.1}
Let $X$ be a smooth algebraic variety defined over $\mathbb C$. 
Then there exists a morphism $\alpha\colon X\to A$ to a quasi-abelian 
variety $A$ such that 
\begin{itemize}
\item[(i)] for any other morphism $\beta\colon X\to B$ to a quasi-abelian 
variety $B$, 
there is a morphism $f\colon A\to B$ such that $\beta=f\circ \alpha$
\begin{equation*}
\xymatrix{
X\ar[d]_{\alpha}\ar[r]^\beta& B\\
A\ar[ur]_{f}&
}
\end{equation*} 
and 
\item[(ii)] $f$ is uniquely determined. 
\end{itemize} 
\end{thm}
 
A quasi-abelian variety in Theorem \ref{p-thm1.1} 
is sometimes called a semi-abelian variety in the 
literature, which is an extension of 
an abelian variety by an algebraic torus as an algebraic group. 
Note that if $X$ is complete in Theorem \ref{p-thm1.1} 
then $A$ is nothing but the Albanese variety of $X$. 
Theorem \ref{p-thm1.1} depends on Deligne's theory of 
mixed Hodge structures for smooth complex algebraic varieties. 

We also give a detailed proof of 
Kawamata's theorem on the quasi-Albanese maps 
for varieties of the logarithmic Kodaira dimension zero. 

\begin{thm}[{see \cite{kawamata-abelian} 
and Theorem \ref{p-thm10.1}}]\label{p-thm1.2}
Let $X$ be a smooth variety such that the logarithmic Kodaira 
dimension $\overline \kappa (X)$ of $X$ is zero. 
Then the quasi-Albanese map $\alpha\colon X\to A$ is dominant and has 
irreducible general fibers. 
\end{thm}

The original proof of Theorem \ref{p-thm1.2} in \cite{kawamata-abelian} 
needs some 
deep results on the theory of variations of (mixed) Hodge structure. 
They are the hardest parts of \cite{kawamata-abelian} to follow. 
In Section \ref{p-sec7}, we give many supplementary comments 
on various semipositivity theorems, which clarify Kawamata's original 
approach to Theorem \ref{p-thm1.2} in \cite{kawamata-abelian}. 
In Section \ref{p-sec8}, we explain how to avoid using 
the theory of variations of (mixed) Hodge structure for 
the proof of Theorem \ref{p-thm1.2}. 
A vanishing theorem in \cite{fujino-higher} 
related to the theory of mixed Hodge structures 
is sufficient for the proof of Theorem \ref{p-thm1.2}. 

When the logarithmic irregularity $\overline q(X)$ 
of $X$ is $\dim X$ in Theorem 
\ref{p-thm1.2}, we have the following theorem, 
which is a slight refinement of \cite[Theorem A]{mendes}. 

\begin{thm}[{see \cite[Theorem A]{mendes}, 
\cite{fmpt}, and \cite[Lemma 3.2]{cdy}}]\label{p-thm1.3}
Let $X$ be a smooth variety with $\overline \kappa (X)=0$ 
and the logarithmic irregularity $\overline q(X)=\dim X$. 
Then the quasi-Albanese map $\alpha \colon X\to A$ 
is birational and there exists a closed subset $Z$ of $A$ 
with $\codim_A Z\geq 2$ such that 
$\alpha\colon X\setminus \alpha^{-1}(Z)\to A\setminus Z$ 
is an isomorphism and $\alpha^{-1}(Z)$ is of 
pure codimension one. 
\end{thm}

As an easy consequence of Theorem \ref{p-thm1.3}, 
we have: 

\begin{cor}[{see \cite[Corollary B]{mendes}}]\label{p-cor1.4}
Let $X$ be a smooth affine variety with $\dim X=n$. 
Then $X$ is isomorphic to $\mathbb G^n_m$ if and only 
if $\overline \kappa (X)=0$ and $\overline q(X)=n$. 
\end{cor}

Corollary \ref{p-cor1.5} is also an easy consequence of 
Theorem \ref{p-thm1.3}. 

\begin{cor}\label{p-cor1.5} 
Let $X$ be a nonempty Zariski open set of a quasi-abelian 
variety $A$. 
Then $\overline \kappa(X)=0$ if and only 
if $\codim _A(A\setminus X)\geq 2$. 
\end{cor}

One of the main motivations of this paper is 
to understand Theorem \ref{p-thm1.2} in detail. 
The original proof of Theorem \ref{p-thm1.2} 
in \cite{kawamata-abelian} looks inaccessible 
because the theory of variations of (mixed) Hodge 
structure was not fully matured when \cite{kawamata-abelian} 
was written around 1980. 
Moreover, some details are omitted in \cite{kawamata-abelian}. 
In \cite{kawamata-abelian}, Kawamata could and did use only 
\cite{deligne}, \cite{griffiths}, and \cite{schmid} for the 
Hodge theory. 
Although the semipositivity theorem in \cite{fujino-higher} 
(see also \cite{fujino-fujisawa} and \cite{ffs}) does not 
recover Kawamata's statement on semipositivity (see 
\cite[Theorem 32]{kawamata-abelian}), it is natural and is 
sufficient for 
us to carry out Kawamata's proof 
of Theorem \ref{p-thm1.2} in \cite{kawamata-abelian} with 
some suitable modifications. 
The author has been unable to follow \cite[Theorem 32]{kawamata-abelian}. 
Moreover, the vanishing theorem in \cite{fujino-higher} 
gives a more elementary 
approach to Theorem \ref{p-thm1.2} and 
makes Theorem \ref{p-thm1.2} independent of 
the theory of variations of (mixed) Hodge structure. 
The author hopes that this paper will make Iitaka's 
theory of quasi-Albanese maps and 
Kawamata's result on the quasi-Albanese maps of 
varieties of the logarithmic Kodaira dimension zero 
accessible. 

We look at the organization of this paper. 
Section \ref{p-sec2} is a preliminary section. 
In Subsections \ref{p-subsec2.1} and 
\ref{p-subsec2.2}, we collect some basic 
definitions and results of 
the logarithmic Kodaira dimensions 
and the quasi-abelian varieties in the sense of 
Iitaka, respectively. 
Section \ref{p-sec3} is devoted to the theory of quasi-Albanese maps and 
varieties due to Shigeru Iitaka. 
We explain it in details following Iitaka's paper \cite{iitaka1} with many 
supplementary 
arguments. Theorem \ref{p-thm3.16}, which is Theorem \ref{p-thm1.1}, 
is the main result of this section. 
In Section \ref{p-sec4}, we prove some basic 
properties of quasi-abelian varieties for the reader's 
convenience. 
In Section \ref{p-sec5}, we quickly explain a birational characterization of 
abelian varieties and a bimeromorphic characterization of 
complex tori without proof. In Section \ref{p-sec6}, we recall 
the subadditivity of the logarithmic Kodaira dimensions 
in some special cases. We use them for the proof of 
Theorem \ref{p-thm1.2}. 
Section \ref{p-sec7} is devoted to the explanation of 
some semipositivity 
theorems related to the theory of variations of (mixed) Hodge structure. 
We hope that this section will help the reader to 
understand \cite{kawamata-abelian}.   
In Section \ref{p-sec8}, we discuss some weak positivity theorems. 
Our approach in Section \ref{p-sec8} does not use the theory of 
variations of (mixed) Hodge structure. 
We use a generalization of the Koll\'ar vanishing theorem. 
This section makes Theorem \ref{p-thm1.2} independent of 
the theory of variations of (mixed) Hodge structure. 
In Section \ref{p-sec9}, 
we discuss finite covers of quasi-abelian varieties. 
We need them for the proof of Theorem \ref{p-thm1.2}. 
In Section \ref{p-sec10}, 
we prove Theorem \ref{p-thm1.2} in details. 
In Section \ref{p-sec11}, which is the final section, 
we prove Theorem \ref{p-thm1.3} and Corollaries  
\ref{p-cor1.4} and \ref{p-cor1.5}. 

\begin{say}[Historical note]\label{p-say1.5} 
If I remember correctly, I wrote the following two preprints: 
\begin{itemize}
\item Osamu Fujino, Subadditivity of the logarithmic 
Kodaira dimension for morphisms of relative dimension one 
revisited 
\end{itemize} 
and 
\begin{itemize}
\item Osamu Fujino, On quasi-Albanese maps 
\end{itemize}
in 2014 and circulated them as Kyoto 
Math 2015-02 and 2015-03 in the preprint series in Department of 
Mathematics, Kyoto University, respectively. 
In 2018, I gave a series of lectures on the Iitaka conjecture in Osaka. 
Then I published \cite{iitaka-conjecture}, 
which is a completely revised and 
expanded version of the above first 
preprint. Although I did not put the above preprints 
on the arXiv, some people have cited the above second 
preprint as a reference on 
the theory of quasi-Albanese maps. 
Hence, I put it on the arXiv now and will plan to publish it somewhere. 
Note that I added Section \ref{p-sec11} 
when I revised this paper in 2024.  
\end{say}

\begin{ack}\label{p-ack}
The author was partially supported by Grant-in-Aid for 
Young Scientists (A) 24684002 and Grant-in-Aid for 
Scientific Research (S) 24224001 from JSPS. 
He thanks Kentaro Mitsui for answering his questions. 
He also would like to thank Professor Noboru 
Nakayama for a useful comment. 
During the preparation of the revised version of this paper, 
he was partially 
supported by JSPS KAKENHI Grant Numbers 
JP19H01787, JP20H00111, JP21H00974, JP21H04994. 
He thanks Professors Margarida Mendes Lopes, Rita Pardini, and 
Sofia Tirabassi very much for helping him to 
understand Theorem \ref{p-thm1.3}. 
Finally, he would like to 
thank Professor Katsutoshi Yamanoi for his  
encouragement.
\end{ack}

We will work over $\mathbb C$, the complex number field, throughout this 
paper. A variety means a reduced and irreducible 
separated scheme of finite type over $\mathbb C$. 
We will use the standard notation as in 
\cite{fujino-fundamental} and \cite{fujino-foundations}. 
The theory of algebraic groups which are not affine nor projective is 
not so easy to access. 
Hence we make efforts to minimize the use of the general theory of 
algebraic groups for 
the reader's convenience. In this paper, we do not even 
use \cite[Lemme (10.1.3.3)]{deligne-hodgeIII}. 
We do not use the theory of minimal models. 

\section{Preliminaries}\label{p-sec2}
In this section, we collect some basic definitions and 
results on the logarithmic Kodaira dimensions and 
the quasi-abelian varieties (see, for 
example, \cite{iitaka1}, \cite{iitaka2}, 
\cite{iitaka4}, \cite{iitaka5}, and so on). 
For the basic properties of the Kodaira dimensions and some related 
topics, see, for example, \cite{ueno} and \cite{mori} 
(see also \cite{iitaka-conjecture}). 

\subsection{Logarithmic Kodaira dimensions and irregularities}
\label{p-subsec2.1} 

First, we recall the logarithmic Kodaira dimensions and the logarithmic 
irregularities following Iitaka. For the details, see \cite{iitaka1}, 
\cite{iitaka2}, \cite{iitaka4}, and \cite{iitaka5}. 

\begin{defn}[Logarithmic Kodaira dimension]\label{p-def2.1}
Let $X$ be an algebraic variety. 
By Nagata (see \cite{nagata}), we have a complete algebraic variety 
$\overline X$ which contains $X$ as a dense Zariski open subset. 
By Hironaka (see \cite{hironaka}), we have a smooth projective 
variety $\overline W$ and a projective 
birational morphism $\mu\colon \overline W\to \overline X$ such that 
if $W=\mu^{-1}(X)$, then $\overline D=\overline W-W=\mu^{-1}(\overline X-X)$ 
is a simple normal crossing divisor on $\overline W$. 
The {\em{logarithmic Kodaira dimension}} $\overline \kappa(X)$ of 
$X$ is defined as 
\begin{equation*} 
\overline \kappa(X)=\kappa(\overline W, K_{\overline {W}}+\overline D) 
\end{equation*} 
where $\kappa$ denotes Iitaka's $D$-dimension. 
\end{defn}

\begin{defn}[Logarithmic irregularity]\label{p-def2.2}
Let $X$ be an algebraic variety. 
We take $(\overline W, \overline D)$ as in Definition \ref{p-def2.1}. 
Then we put 
\begin{equation*}
\overline q(X)=\dim _{\mathbb C}
H^0(\overline W, \Omega_{\overline W}^1(\log \overline D))
\end{equation*} 
and call it the {\em{logarithmic irregularity}} of $X$.  
We put 
\begin{equation*} 
T_1(X)=H^0(\overline W, \Omega^1_{\overline W}(\log \overline D))
\end{equation*}  
following Iitaka \cite{iitaka1}. 
\end{defn}

It is easy to see: 

\begin{lem}\label{p-lem2.3}
$\overline \kappa(X)$, $\overline q(X)$, and 
$T_1(X)$  
are well-defined, that is, 
they are independent of the choice of the pair $(\overline W, \overline D)$. 
\end{lem}

This lemma is well known. We give a proof for the reader's convenience. 
\begin{proof}
By Hironaka's resolution (see \cite{hironaka}), 
it is sufficient to 
prove that 
\begin{equation*}
\kappa(\overline W, K_{\overline W}+\overline D)=\kappa 
(\overline W_1, K_{\overline W_1}+\overline D_1) 
\end{equation*}  
and 
\begin{equation*} 
H^0(\overline W, \Omega^1_{\overline W}(\log \overline D))
=H^0(\overline W_1, \Omega^1_{\overline W_1}(\log \overline D_1))
\end{equation*} 
where $f\colon \overline W_1\to \overline W$ 
is a projective 
birational morphism from a smooth projective variety $\overline W_1$ and 
$\overline D_1=\Supp f^*\overline D$. 
By the local calculation, we see that 
\begin{equation}\label{p-eq2.1}
f^*\Omega^1_{\overline W}(\log \overline D)\subset \Omega^1_{\overline W_1}
(\log \overline D_1). 
\end{equation}
Therefore, we obtain 
\begin{equation*} 
H^0(\overline W, \Omega^1_{\overline W}(\log \overline D))
\subset 
H^0(\overline W_1, \Omega^1_{\overline W_1}(\log \overline D_1)). 
\end{equation*}  
On the other hand, it is obvious that 
\begin{equation*} 
H^0(\overline W_1, \Omega^1_{\overline W_1}(\log \overline D_1))
\subset 
H^0(\overline W, \Omega^1_{\overline W}(\log \overline D))
\end{equation*}  
since $\Omega^1_{\overline W}(\log \overline D)$ is locally free. 
Thus, we have 
\begin{equation*} 
H^0(\overline W_1, \Omega^1_{\overline W_1}(\log \overline D_1))=
H^0(\overline W, \Omega^1_{\overline W}(\log \overline D)). 
\end{equation*}  
By \eqref{p-eq2.1}, we have 
\begin{equation*} 
K_{\overline W_1}+\overline D_1=f^*(K_{\overline W}+\overline D)+E
\end{equation*}  
where $E$ is an effective $f$-exceptional 
divisor on $\overline W_1$. 
Therefore, it is obvious that 
\begin{equation*} 
\kappa (\overline W, K_{\overline W}+\overline D)=\kappa (\overline W_1, 
K_{\overline W_1}+\overline D_1)
\end{equation*}  
holds. 
\end{proof}

\begin{lem}\label{p-lem2.4}
Let $X$ be a variety and let $U$ be a nonempty Zariski open set 
of $X$. 
Then we have 
\begin{equation*} 
\overline \kappa (X)\leq \overline \kappa (U)
\end{equation*}  
and 
\begin{equation*} 
\overline q(X)\leq \overline q(U). 
\end{equation*} 
\end{lem}

\begin{proof}
It is obvious by the definitions of $\overline \kappa$ and 
$\overline q$. 
\end{proof}

We may sometimes use Lemma \ref{p-lem2.4} implicitly. 
We need Lemma \ref{p-lem2.5} for the proof of 
Corollary \ref{p-cor1.5}. 

\begin{lem}\label{p-lem2.5} 
Let $X$ be a smooth variety. 
Assume that $F$ is a closed subset of $X$ with 
$\codim_XF\geq 2$. 
Then we have 
\begin{equation*} 
\overline\kappa(X)=\overline \kappa (X-F).  
\end{equation*} 
\end{lem}
\begin{proof}
We take a smooth complete algebraic variety $\overline X$ such that 
$\overline D=\overline X-X$ is a simple 
normal crossing divisor on $\overline X$. 
Then we have 
\begin{equation*} 
\overline \kappa (X)=\kappa (\overline X, K_{\overline X}+\overline D)
\end{equation*}  
by definition. 
Let $\overline F$ be the closure of $F$ in $\overline X$. 
Note that $\codim_{\overline X}\overline F\geq 2$. 
We take a resolution 
\begin{equation*} 
f\colon Y\to \overline X 
\end{equation*}  
such that $f$ is an isomorphism 
over $\overline X-\overline F$ and that 
$\Supp f^{-1}(\overline F)$ and $\Supp 
\left(f^{-1}(\overline F)\cup f^*\overline D\right)$ 
are simple normal crossing 
divisors on $Y$. 
We put $\Delta_1=\Supp f^*\overline D$ and 
$\Delta_2=\Supp \left( f^{-1}(\overline F)\cup f^*\overline D\right)$. 
Then we have 
\begin{equation*} 
K_Y+\Delta_1=f^*(K_{\overline X}+\overline D)+E
\end{equation*}  
where $E$ is an effective $f$-exceptional 
divisor on $Y$. Therefore, 
we obtain 
\begin{equation*} 
\overline \kappa (X)=\kappa (\overline X, K_{\overline X}+\overline D)
=\kappa (Y, K_Y+\Delta_1). 
\end{equation*}  
By definition, 
\begin{equation*} 
\overline \kappa (X-F)=\kappa (Y, K_Y+\Delta_2).
\end{equation*}  
Since $\Delta_2-\Delta_1$ is an effective 
$f$-exceptional divisor on $Y$ by 
$\codim_{\overline X}\overline F\geq 2$, 
we have 
\begin{align*}
\overline \kappa (X-F)=\kappa (Y, K_Y+\Delta_2)=\kappa 
(\overline X, K_{\overline X}+\overline D)=\overline \kappa (X). 
\end{align*}
This is the desired equality. 
\end{proof}

We will freely use the following lemmas throughout this paper. 

\begin{lem}\label{p-lem2.6}
Let $f\colon X\to Y$ be a dominant morphism of 
algebraic varieties. 
Then $\overline \kappa (X)\geq \overline \kappa (Y)$ holds. 
\end{lem}

\begin{proof} 
By using Hironaka's resolution of singularities, 
we may assume that $X$ and $Y$ are both smooth. 
Let $\overline f\colon \overline X\to \overline Y$ be a 
compactification of $f\colon X\to Y$, 
that is, $\overline X$ and $\overline Y$ are 
smooth complete varieties and $\Delta_{\overline X}:=
\overline X-X$ and $\Delta_{\overline Y}:=\overline Y-Y$ 
are simple normal crossing divisors on $\overline X$ and 
$\overline Y$, respectively. 
Since 
\begin{equation*}
\overline f^*\Omega^1_{\overline Y}(\log \Delta_{\overline Y})
\subset \Omega^1_{\overline X}(\log \Delta_{\overline X}), 
\end{equation*} 
we have 
\begin{equation*}
K_{\overline X}+\Delta_{\overline X}=\overline f^*(K_{\overline Y}+
\Delta_{\overline Y})+E, 
\end{equation*} 
where $E$ is effective. 
This implies that 
\begin{equation*}
\overline \kappa (Y)=\kappa (\overline Y, K_{\overline Y}+
\Delta_{\overline Y})=\kappa (\overline X, 
\overline f^*(K_{\overline Y}+
\Delta_{\overline Y}))\leq 
\kappa (\overline X, K_{\overline X}+\Delta_{\overline X})
=\overline \kappa (X). 
\end{equation*} 
We finish the proof. 
\end{proof}

\begin{lem}[{\cite[Theorem 3]{iitaka2}}]\label{p-lem2.7} 
Let $f\colon X\to Y$ be a finite \'etale morphism 
of algebraic varieties. 
Then $\overline \kappa (X)=\overline \kappa (Y)$ holds. 
\end{lem}

\begin{proof} 
Let $\widetilde Y\to Y$ be a resolution of singularities 
of $Y$. We replace $Y$ and $X$ with $\widetilde Y$ and 
$X\times _Y\widetilde Y$, respectively. 
Then we may assume that $X$ and $Y$ are smooth. 
Let $\overline f\colon \overline X\to \overline Y$ be 
a compactification of $f\colon X\to Y$, that is, 
$\overline X$ and $\overline Y$ are smooth complete varieties and 
$\Delta_{\overline X}:=\overline X-X$ and 
$\Delta_{\overline Y}:=\overline 
Y-Y$ are simple normal crossing divisors on $\overline X$ and $\overline Y$, 
respectively. Let 
\begin{equation*}
\xymatrix{
\overline f\colon \overline X \ar[r]^-{\overline g}
& \overline Z\ar[r]^-{\overline h}&\overline Y
}
\end{equation*}
be the Stein factorization of $\overline f$. 
We note that $\overline h$ is finite 
and $\overline g$ is birational. 
We put $Z:=\overline h^{-1}(Y)$ and 
$\Delta_{\overline Z}:=\overline Z-Z$. 
Then we can easily check that 
\begin{equation*}
K_{\overline Z}+\Delta_{\overline Z}=\overline h^*(K_{\overline Y}+
\Delta_{\overline Y})
\end{equation*} 
and 
\begin{equation*}
K_{\overline X}+\Delta_{\overline X}=\overline g^*(K_{\overline Z}+
\Delta_{\overline Z})+E, 
\end{equation*} 
where $E$ is effective and $\overline g$-exceptional. 
Thus, we have 
\begin{equation*}
\overline \kappa (X)=
\kappa (\overline X, K_{\overline X}+\Delta_{\overline X})
=\kappa (\overline Z, K_{\overline Z}+\Delta_{\overline Z})
=\kappa (\overline Y, K_{\overline Y}+\Delta_{\overline Y})
=\overline \kappa (Y). 
\end{equation*} 
This is what we wanted. 
\end{proof}

\subsection{Quasi-abelian varieties in the sense 
of Iitaka}\label{p-subsec2.2}
From now on, 
we quickly recall the basic 
properties of quasi-abelian varieties in the sense of 
Iitaka (see \cite{iitaka1} and \cite{iitaka2}). 
This paper shows that Iitaka's definition of quasi-abelian varieties 
is reasonable and natural 
from the viewpoint of the birational geometry. 

\begin{defn}[Quasi-abelian varieties in the sense of 
Iitaka]\label{p-def2.8} 
Let $G$ be a connected algebraic group. Then we have the 
Chevalley decomposition 
(see, for example, \cite[Theorem 1.1]{conrad} 
and \cite[Theorem 1.1.1]{brion-etal}): 
\begin{equation*} 
1\to \mathcal G\to G\to \mathcal A\to 1
\end{equation*}  
in which $\mathcal G$ is the maximal 
affine algebraic subgroup of $G$ and 
$\mathcal A$ is an abelian variety. 
If $\mathcal G$ is an algebraic torus 
$\mathbb G_m^d$ of dimension $d$, 
then $G$ is called a {\em{quasi-abelian}} variety (in the sense of Iitaka). 
\end{defn}

We make some important remarks. 

\begin{rem}\label{p-rem2.9} 
A quasi-abelian variety in Definition \ref{p-def2.8} is 
sometimes called a {\em{semi-abelian variety}} 
in the literature (see, 
for example, \cite{brion-etal}, 
\cite[Definition 5.1.20]{noguchi-winkelmann}, 
\cite[Appendix A.~Semi-abelian varieties]{yamanoi}, 
and so on). 
We also note that the Chevalley decomposition in 
Definition \ref{p-def2.8} is called 
{\em{Chevalley's structure theorem}} 
in \cite{brion-etal}. 
\end{rem}

\begin{rem}\label{p-rem2.10} 
Let $G$ be a quasi-abelian variety, that is, $\mathcal G$ is an 
algebraic torus $\mathbb G^d_m$ in 
Definition \ref{p-def2.8}. 
Then, it is known that 
$G$ is a principal $\mathbb G^d_m$-bundle 
over $\mathcal A$ in the Zariski topology 
(see, for example, \cite[Theorems 4.4.1 and 4.4.2]{bcm}). 
\end{rem}

Note that the definition of quasi-abelian varieties in the sense of 
Iitaka (see Definition \ref{p-def2.8}) is 
different from the definition in \cite[3.~Quasi-Abelian Varieties]{abe-k}. 
We also note that if $G$ is a quasi-abelian variety in the sense of 
Iitaka then $G$ is a quasi-abelian variety in the sense of \cite{abe-k} 
(see, for example, \cite[3.2.21 Main Theorem]{abe-k}). 

\begin{rem}\label{p-rem2.11} It is well known that 
every algebraic group is quasi-projective 
(see, for example, \cite[Corollary 1.2]{conrad}). 
\end{rem}

Although we do not need the following fact, 
we can easily check: 

\begin{rem}\label{p-rem2.12}
Let $G$ be a connected algebraic group. Then 
$G$ is a quasi-abelian variety if and only if $G$ contains 
no $\mathbb G_a$ as an algebraic subgroup (see \cite[Lemma 3]{iitaka2}). 
\end{rem}

We note the following important property. 

\begin{lem}[{see \cite[Lemma 4]{iitaka2}}]\label{p-lem2.13}
A quasi-abelian variety is a commutative algebraic group. 
\end{lem}
 
\begin{proof}
We take $\tau\in G$ and consider the group homomorphism: 
\begin{equation*} 
\Psi_\tau(\sigma)=\tau\sigma\tau^{-1}\colon \mathcal G\to G. 
\end{equation*}  
Since $\mathcal G$ is rational and $\mathcal A$ is 
an abelian variety, we see that $\Psi_\tau\colon  \mathcal G\to \mathcal G$. 
Therefore, we obtain 
\begin{equation*} 
G\ni \tau \mapsto \Psi_\tau\in \mathrm{Hom}(\mathcal G, \mathcal G). 
\end{equation*}  
Note that $\mathrm{Hom}(\mathcal G, \mathcal G)$ is discrete 
because $\mathcal G$ is an algebraic torus. 
Thus, we obtain $\Psi_1=\Psi_\tau$. 
Therefore, $\mathcal G$ is contained in the center of $G$. 
Moreover, if $\sigma, \tau\in G$, then 
we have 
\begin{equation*} 
[\tau, \sigma]=\tau\sigma\tau^{-1}\sigma^{-1}\in \mathcal G
\end{equation*}  
since $\mathcal A$ is commutative. 
Let $\rho$ be any element of 
$\mathcal G$. 
Then it is easy to see that 
\begin{equation*} 
[\tau\rho, \sigma]=[\tau, \sigma] 
\end{equation*}  
since $\mathcal G$ is contained in the center of $G$. 
Note that $G$ is a principal 
$\mathcal G$-bundle over $\mathcal A$ as a complex manifold. 
Therefore, the morphism 
\begin{equation*} 
G\ni \tau\mapsto [\tau, \sigma]\in \mathcal G
\end{equation*}  
factors through a holomorphic 
map 
\begin{equation*} 
\mathcal A\to \mathcal G, 
\end{equation*}  
which is obviously trivial since $\mathcal A$ is complete. 
Hence, we obtain 
\begin{equation*} 
[\tau, \sigma]=1
\end{equation*}  
for every $\sigma, \tau\in G$. 
This implies that $G$ is commutative. 
\end{proof}

\begin{rem}\label{p-rem2.14}
Let $G$ be a quasi-abelian variety. By Lemma \ref{p-lem2.13}, $G$ is 
a commutative group. 
Therefore, from now on, we write the group law in $G$ additively if 
there is no danger of confusion. 
The unit element of $G$ is denoted by $0$. 
Note that an algebraic torus $\mathbb G^d_m$ is a quasi-abelian variety 
in the sense of Iitaka. 
\end{rem}

\begin{lem}\label{p-lem2.15}
Let $A$ be a quasi-abelian variety and let $B$ be an algebraic 
subgroup of $A$. 
Then $B$ and $A/B$ are quasi-abelian varieties. 
\end{lem}

\begin{proof}
We can construct the following big commutative diagram 
of Chevalley decompositions. 
\begin{equation*}
\xymatrix{& 0\ar[d]& 0\ar[d]& 0\ar[d]&\\
0\ar[r]& \mathcal G_B\ar[r]\ar[d]& B\ar[d]\ar[r] & \mathcal A_B
\ar[d]\ar[r]
\ar[r]\ar[d]&0\\
0\ar[r]& \mathcal G_A\ar[d]\ar[r]& A \ar[r]\ar[d]& \mathcal A_A
\ar[r]\ar[d]&0\\
0\ar[r]& \mathcal G_{A/B}\ar[d]\ar[r]& A/B
\ar[d]\ar[r] & \mathcal A_{A/B}\ar[d]\ar[r]&0\\ 
& 0& 0& 0&
}
\end{equation*} 
Note that subgroups and homomorphic images of 
an algebraic torus are again algebraic tori. 
Hence we see that $B$ and $A/B$ are quasi-abelian varieties. 
We finish the proof. 
\end{proof}

We sometimes treat quasi-abelian varieties as commutative complex 
Lie groups. 

\begin{lem}\label{p-lem2.16}
Let $G$ be a quasi-abelian variety. 
Then the universal cover of $G$ is $\mathbb C^{\dim G}$ and 
$G$ is $\mathbb C^{\dim G}/L$ for some lattice $L$ as a complex 
Lie group. 
Of course, $L$ is nothing but the topological fundamental group $\pi_1(G)$ 
of $G$. 
Note that the group law of $G$ is induced by the 
usual addition of $\mathbb C^{\dim G}$.  
\end{lem}

\begin{proof}
By Lemma \ref{p-lem2.13}, $G$ is a commutative complex 
Lie group. Therefore, the universal cover is $\mathbb C^{\dim G}$ and
there is a discrete subgroup $L$ of $\mathbb C^{\dim G}$ 
such that $G=\mathbb C^{\dim G}/L$ as a complex Lie group. 
By construction, the group law in $G$ is induced by 
the usual addition of $\mathbb C^{\dim G}$ 
(see also the proof of Lemma \ref{p-lem3.8}). 
\end{proof}

In this paper, we mainly treat non-projective 
algebraic groups as complex Lie groups. 
We note the following famous example. 
It says that two different algebraic groups may be analytically isomorphic. 
Of course, we can not directly use Serre's GAGA principle for 
non-projective varieties. 

\begin{ex}[Vector extensions of elliptic curves]\label{p-ex2.17} 
Let $E$ be an elliptic curve. 
We take $0\ne \xi\in H^1(E, \mathcal O_E)\simeq \mathbb C$. 
Then we have a non-trivial $\mathbb G_a$-bundle 
$G$ over $E$ associated to $\xi$ in the Zariski topology. 
It is well known that there exists a short exact sequence 
of commutative algebraic groups 
\begin{equation}\label{p-eq2.2}
0\to \mathbb G_a\to G\to E\to 0, 
\end{equation} 
that is, $G$ is a commutative algebraic group which 
is an extension of $E$ by $\mathbb G_a$. 
Since $\xi \ne 0$, we can check that 
$G$ is analytically isomorphic to $(\mathbb C^\times)^2$. 
Hence, $\mathbb G^2_m$ and $G$ are analytically 
isomorphic but are two different algebraic groups. 
Note that \eqref{p-eq2.2} is the Chevalley decomposition 
of $G$. For some related topics, see \cite{brion-etal}, 
\cite[Chapter VI, Example 3.2]{hartshorne}, 
\cite[Footnote in page 33]{mumford}, and so on. 
\end{ex}

In Section \ref{p-sec3}, 
we will discuss Iitaka's quasi-Albanese maps and 
prove the existence of quasi-Albanese maps and varieties in details. 

\begin{defn}[Quasi-Albanese maps]\label{p-def2.18}
Let $X$ be a smooth variety. 
The {\em{quasi-Albanese map}} $\alpha\colon X\to A$ is a morphism to a quasi-abelian 
variety $A$ such that 
\begin{itemize}
\item[(i)] for any other morphism $\beta\colon X\to B$ to a quasi-abelian 
variety $B$, 
there is a morphism $f\colon A\to B$ such that $\beta=f\circ \alpha$
\begin{equation*} 
\xymatrix{
X\ar[d]_{\alpha}\ar[r]^\beta& B\\
A\ar[ur]_{f}&
}
\end{equation*} 
and 
\item[(ii)] $f$ is uniquely determined. 
\end{itemize} 
Note that $A$ is usually called the {\em{quasi-Albanese variety}} 
of $X$. 
\end{defn} 
If $X$ is complete in Definition \ref{p-def2.18}, then 
$A$ is nothing but the Albanese variety of $X$. 

\section{Quasi-Albanese maps due to Iitaka}\label{p-sec3}

In this section, we discuss Iitaka's quasi-Albanese maps and 
varieties following \cite{iitaka1} and \cite{iitaka2}. 
We recommend the reader to study the 
basic results on the Albanese maps and 
varieties before reading 
this section (see, for example, \cite[V.11--14]{beauville}, \cite[\S 9]{ueno}, 
\cite[Chapter 2, Section 6]{griffiths-harris},  
and so on). 

Let us start with the following easy lemma on singular homology groups. 
In this section, $X$ is a smooth complete algebraic variety and $D$ is 
a simple normal crossing divisor on $X$. The Zariski open set 
$X\setminus D$ of $X$ is denoted by 
$V$. 

\begin{lem}\label{p-lem3.1}
Let $X$ be a smooth complete algebraic variety and let 
$D$ be a simple normal crossing divisor on $X$. 
We put $V=X\setminus D$. 
Then the map 
\begin{equation*} 
\iota_*\colon  H_1(V, \mathbb Z)\to H_1(X, \mathbb Z)
\end{equation*}  
is surjective, where $\iota\colon V\hookrightarrow X$ is the natural open immersion. 
\end{lem}

\begin{proof}
We put $n=\dim X$. 
We have the following long exact sequence: 
\begin{equation*} 
\cdots \longrightarrow 
H^{2n-1}(X, D; \mathbb Z)\overset {p}{\longrightarrow} 
H^{2n-1}(X, \mathbb Z)\longrightarrow 
H^{2n-1}(D, \mathbb Z)\longrightarrow \cdots. 
\end{equation*} 
Note that $H^{2n-1}(D, \mathbb Z)=0$ since $D$ is an $(n-1)$-dimensional 
simple normal crossing variety. 
Therefore, $p$ is surjective. 
We also note that 
\begin{equation*} 
H^{2n-1}(X, D; \mathbb Z)\simeq H^{2n-1}_c(V, \mathbb Z). 
\end{equation*} 
We have the following commutative diagram: 
\begin{equation*} 
\xymatrix
{
H^{2n-1}_c(V, \mathbb Z)\ar[d]_{D_V}^{\simeq}
\ar[r]^{p}&
H^{2n-1}(X, \mathbb Z)\ar[d]^{D_X}_{\simeq}\\
H_1(V, \mathbb Z)\ar[r]_{\iota_*} &
H_1(X, \mathbb Z). 
}
\end{equation*} 
Note that the duality maps $D_V$ and $D_X$ are both isomorphisms 
by Poincar\'e duality. 
Since $p$ is surjective, we see that 
$\iota_*$ is also surjective. 
For the details of Poincar\'e duality, see, for example, 
\cite[Section 3.3]{hatcher}. 
\end{proof}

\begin{lem}\label{p-lem3.2}
The natural injection 
\begin{equation*} 
\iota^*\colon  H^1(X, \mathbb C)\to H^1(V, \mathbb C) 
\end{equation*}  
is nothing but 
\begin{equation*} 
a_1\oplus a_2\colon  H^1(X, \mathcal O_X)\oplus H^0(X, \Omega^1_X)
\to H^1(X, \mathcal O_X)\oplus H^0(X, \Omega^1_X(\log D))
\end{equation*} 
where 
$a_1$ is the identity on $H^1(X, \mathcal O_X)$ and 
$a_2$ is the natural inclusion 
\begin{equation*} 
H^0(X, \Omega^1_X)\hookrightarrow H^0(X, \Omega^1_X(\log D))
\end{equation*} 
by Deligne's theory of mixed Hodge 
structures. Note that we have 
\begin{equation*} 
b_1(V)-b_1(X)=\overline q(V)-q(X)
\end{equation*}  
where 
\begin{equation*} 
\overline q(V)=\dim H^0(X, \Omega^1_X(\log D)) 
\quad \text{and} \quad q(X)=\dim H^0(X, \Omega^1_X). 
\end{equation*} 
Of course, 
\begin{equation*} 
b_1(V)=\dim_\mathbb C H^1(V, \mathbb C) \quad 
\text{and}\quad b_1(X)=\dim _\mathbb C H^1(X, \mathbb C). 
\end{equation*} 
\end{lem}

\begin{proof}
By Lemma \ref{p-lem3.1}, $\iota^*$ is injective. 
By Deligne's mixed Hodge theory (see \cite{deligne}), we have 
\begin{equation*} 
H^1(V, \mathbb C)=H^1(X, \mathcal O_X)\oplus H^0(X, \Omega^1_X(\log D)). 
\end{equation*} 
By the Hodge decomposition, we have 
\begin{equation*} 
H^1(X, \mathbb C)=H^1(X, \mathcal O_X)\oplus H^0(X, \Omega^1_X).
\end{equation*}  
Since $\iota^*\colon  H^1(X, \mathbb C)\to H^1(V, \mathbb C)$ is a morphism 
of mixed Hodge structures 
(see \cite{deligne}), we obtain the desired description 
of $\iota^*$. 
\end{proof}

Let us describe the theory of quasi-Albanese maps and 
varieties due to Shigeru Iitaka (see \cite{iitaka1}). 

\begin{say}[Quasi-Albanese maps and varieties]\label{p-say3.3}
We take a basis 
\begin{equation*} 
\{\omega_1, \cdots, \omega_q\}
\end{equation*}  
of $H^0(X, \Omega^1_X)$, where $q=q(X)=\dim H^0(X, \Omega^1_X)$. 
Note that $b_1(X)=2q$ by the Hodge theory. 
We take 
\begin{equation*} 
\varphi_1, \cdots, \varphi_d \in H^0(X, \Omega^1_X(\log D))
\end{equation*}  
with $d=\overline q(V)-q(X)$ such that 
\begin{equation*} 
\{\omega_1, \cdots, \omega_q, \varphi_1, \cdots, \varphi_d\}
\end{equation*}  
is a basis of $H^0(X, \Omega^1_X(\log D))$. 
Let 
\begin{equation*} 
\{\xi_1, \cdots, \xi_{2q}\}
\end{equation*}  
be a basis of the free part of $H_1(X, \mathbb Z)$. We take 
\begin{equation*}
\eta_1, 
\cdots, \eta_d\in \mathrm{Ker}\iota_*\subset 
H_1(V, \mathbb Z)
\end{equation*} 
such that 
\begin{equation*} 
\{\xi_1, \cdots, \xi_{2q}, 
\eta_1, \cdots, \eta_d\}
\end{equation*}  
is a basis of the free part of 
$H_1(V, \mathbb Z)$ (see Lemma \ref{p-lem3.1}). 
We put $\overline q=\overline q(V)$, 
\begin{equation*} 
A_i=\left( \int _{\xi_i} \omega_1, \cdots 
\int _{\xi_i} \omega_{q}, 
\int _{\xi_i} \varphi_1, \cdots, \int _{\xi_i} \varphi_d
\right)\in \mathbb C^{\overline q}
\end{equation*}  
for $1\leq i\leq 2q$, and 
\begin{equation*} 
B_j=\left( \int _{\eta_j} \omega_1, \cdots 
\int _{\eta_j} \omega_{q}, 
\int _{\eta_j} \varphi_1, \cdots, \int _{\eta_j} \varphi_d
\right)\in \mathbb C^{\overline q}
\end{equation*}  
for $1\leq j\leq d$. 

\begin{lem}\label{p-lem3.4}
Let $\gamma$ be a torsion element of $H_1(V, \mathbb Z)$. 
Then we have 
\begin{equation*}
\int _\gamma \omega=0
\end{equation*}  
for every $\omega\in H^0(X, \Omega^1_X(\log D))$. 
\end{lem} 
\begin{proof}
It is obvious since 
\begin{equation*} 
m\int _\gamma \omega=\int _{m\gamma}\omega=0 
\end{equation*}  
if $m\gamma =0$ in $H_1(V, \mathbb Z)$. 
\end{proof}

\begin{lem}\label{p-lem3.5} We have 
\begin{equation*} 
\int _{\eta _j}\omega_k=0
\end{equation*} 
for every $j$ and $k$. 
\end{lem}

\begin{proof}
We see that 
\begin{equation*} 
\int _{\eta_j}\omega_k=
\int _{\eta_j} \iota^*\omega_k=\int _{\iota_* \eta_j}
\omega_k=0
\end{equation*} 
since $\iota_*\eta_j=0$. 
\end{proof}

\begin{lem}[{see \cite[Lemma 2]{iitaka1}}]\label{p-lem3.6}
Let $\varphi$ be an arbitrary element of $H^0(X, \Omega^1_X(\log D))$. 
Assume 
\begin{equation*}
\int _\eta \varphi=0
\end{equation*}  
for every $\eta\in \mathrm{Ker \iota_*}\subset H_1(V, \mathbb Z)$. 
Then we have $\varphi \in H^0(X, \Omega^1_X)$. 
\end{lem}

\begin{proof}
Assume that $\varphi\in 
H^0(X, \Omega^1_X(\log D))\setminus H^0(X, \Omega^1_X)$. 
Then $\varphi$ has a pole along some $D_a$, where 
$D_a$ is an irreducible 
component of $D$. 
Let $p$ be a general point of $D_a$. 
We take a local holomorphic coordinate system 
$(z_1, \cdots, z_n)$ around $p$ such that 
$D_a$ is defined by $z_1=0$. 
In this case, we can write 
\begin{equation*} 
\varphi=\alpha(z)\frac{dz_1}{z_1}+\beta(z)
\end{equation*}  
around $p$, 
where $\beta(z)$ is a holomorphic $1$-form. 
We may assume that 
$\alpha(z)=\alpha(z_2, \cdots, z_n)$ by Weierstrass division theorem 
(see, for example, \cite{griffiths-harris}). 
Since $d\varphi=0$, we obtain 
\begin{equation*} 
d\varphi =d\alpha \wedge \frac{dz_1}{z_1}+d\beta=0. 
\end{equation*}  
Thus we have $d\alpha=0$. 
This means that $\alpha$ is a constant. 
Let us consider a circle $\gamma_a$ around $D_a$ at $p$. 
Then we obtain $\iota_*\gamma_a=0$ in $H_1(X, \mathbb Z)$ and 
\begin{equation*} 
0=\int _{\gamma _a}\varphi
=\alpha\int _{\gamma_a} \frac{dz_1}{z_1}=\alpha 2\pi\sqrt{-1}. 
\end{equation*}  
This implies that $\alpha=0$. 
Thus, $\varphi$ is holomorphic at $p$. 
This is a contradiction. 
Therefore, we have $\varphi\in H^0(X, \Omega^1_X)$. 
\end{proof}

\begin{lem}[{see \cite[Proposition 2]{iitaka2}}]\label{p-lem3.7} 
The above vectors $A_1, \cdots, A_{2q}$, 
$B_1, \cdots, B_d$ are $\mathbb R$-$\mathbb C$ 
linearly independent. 
This means that 
\begin{equation*} 
\sum _{i=1}^{2q} a_i A_i +\sum _{j=1}^{d} b_j B_j=0
\end{equation*}  
for $a_i \in \mathbb R$ and $b_j \in \mathbb C$ then 
$a_i=0$ for every $i$ and $b_j=0$ for every $j$. 
\end{lem}
\begin{proof}
We put 
\begin{equation*} 
\widehat A_i =\left( \int _{\xi_i} 
\omega_1, \cdots, \int _{\xi_i}\omega_q\right) 
\end{equation*} 
for $1\leq i\leq 2q$. 
Then $\widehat A_1, \cdots, 
\widehat A_{2q}$ are $\mathbb R$-linearly independent, 
which is well known by the Hodge theory. 
By Lemma \ref{p-lem3.5}, we have  $a_i=0$ for every $i$. 
We put 
\begin{equation*} 
\widehat {B} _j =\left( \int _{\eta_j} \varphi_1, \cdots, 
\int _{\eta_j}\varphi_d\right)
\end{equation*}  
for $1\leq j\leq d$. 
It is sufficient to prove that $\widehat B_1, \cdots, \widehat B_d$ 
are $\mathbb C$-linearly independent. 
If $\widehat B_1, \cdots, \widehat B_d$ are $\mathbb C$-linearly 
dependent, then the rank of the $d\times d$ matrix 
\begin{equation*} 
\left(\int_{\eta_j}\varphi _i\right)_{i, j}
\end{equation*}  
is less than $d$. 
This means that 
there is $(c_1, \cdots, c_d)\ne 0$ such that 
\begin{equation*} 
\int_{\eta_j}\sum _{i=1}^dc_i \varphi_i=0
\end{equation*}  
for every $j$. 
Therefore, we see that 
\begin{equation*} 
\sum _{i=1}^dc_i \varphi_i\in H^0(X, \Omega^1_X)
\end{equation*} 
by Lemma \ref{p-lem3.6}. 
This contradicts the choice of $\{\varphi_1, \cdots, \varphi_d\}$. 
Thus, $\widehat B_1, \cdots, \widehat B_d$ are $\mathbb C$-linearly 
independent. 
\end{proof}

By the proof of Lemma \ref{p-lem3.7}, 
we can choose 
$\varphi_1, \cdots, \varphi_d$ such that 
\begin{equation*} 
\int _{\eta_j} \varphi_k =\delta _{jk}. 
\end{equation*} 
We put 
\begin{equation*} 
L=\sum _i \mathbb Z A_i +\sum _j \mathbb Z B_j, 
\end{equation*} 
\begin{equation*} 
L_1=\sum _i \mathbb Z \widehat A_i,  
\end{equation*}  
and 
\begin{equation*} 
L_0=\sum _j \mathbb Z \widehat B_j. 
\end{equation*}  
Then we get the following short exact sequence of complex 
Lie groups: 
\begin{align}\label{p-eq3.1}
0\longrightarrow \mathbb C^d/{L_0}\longrightarrow \mathbb C^{\overline q}/
L\longrightarrow \mathbb C^q/L_1\longrightarrow 0. 
\end{align} 
Note that $T=\mathbb C^d /{L_0}$ is an algebraic torus $\mathbb G^d_m$ and 
that $\mathcal A_X=\mathbb C^q/L_1$ is the Albanese 
variety of $X$. 
More explicitly, if $(z_1, \cdots, z_d)$ is the standard 
coordinate system of $\mathbb C^d$, 
then 
the isomorphism 
\begin{equation*} 
\mathbb C^d/L_0\overset{\sim}\longrightarrow \mathbb G^d_m
\end{equation*}  
is given by 
\begin{equation*} 
(z_1, \cdots, z_d)\mapsto (\exp 2\pi \sqrt{-1} z_1, \cdots, 
\exp 2\pi\sqrt{-1} z_d). 
\end{equation*} 
We call 
\begin{equation*} 
\widetilde {\mathcal A}_V=\mathbb C^{\overline q}/L
\end{equation*}  
the quasi-Albanese variety of $V$. 
By the above description, we see that 
$\widetilde {\mathcal A}_V$ is a principal 
$\mathbb G^d_m$-bundle over an abelian variety $\mathcal A_X$ as 
a complex manifold. 
We have to check: 

\begin{lem}\label{p-lem3.8}
The quasi-Albanese variety $\widetilde {\mathcal A}_V$ is a 
quasi-abelian variety.  
\end{lem}
\begin{proof}
We put $A=\widetilde {\mathcal A} _V$ and $B=\mathcal A_X$ 
for simplicity. Note that $A$ is a principal $\mathbb G^d_m$-bundle over 
$B$ as a complex manifold. 
We consider the following group homomorphism: 
\begin{equation*} 
\rho\colon  \mathbb G^d_m\to \mathrm{PGL}(d, \mathbb C)
\end{equation*}  
given by 
\begin{equation*} 
\rho (\lambda_1, \cdots, \lambda_d)=\begin{pmatrix} 
1&&& \\
&\lambda_1&&\\
&&\ddots &\\
&&&\lambda_d
\end{pmatrix}. 
\end{equation*}  
By $\rho$, we obtain $\mathbb P^d$-bundle 
$Z=A\times _\rho \mathbb P^d$ over $B=A/\mathbb G^d_m$ which 
is a compactification of $A$. 
It is easy to see that the divisor $\Delta=Z\setminus A$ is a simple 
normal crossing divisor on $Z$ and is ample over $B$. 
Moreover, we can easily see that 
$Z\to B$ and $A\to B$ are locally trivial 
in the Zariski topology. 
From now, we will see that the multiplication  
\begin{equation*} 
\psi\colon  A\times A\longrightarrow A 
\end{equation*} 
of $A$ as a complex Lie group is algebraic. 
By construction, the map $\psi$ can be extended to holomorphic maps 
\begin{equation*} 
g_1\colon  Z\times A\longrightarrow Z \quad \text{and} \quad 
g_2\colon  A\times Z\longrightarrow Z
\end{equation*} 
since $Z$ is a $\mathbb G^d_m$-equivariant embedding of $A$. 
Therefore, 
we obtain a holomorphic map 
\begin{equation*} 
g\colon  Z\times Z\setminus 
\Sigma\longrightarrow Z\hookrightarrow \mathbb P^N, 
\end{equation*}  
where 
$\Sigma=(\Delta\times Z)
\cap (Z\times \Delta)$. 
Of course, $g$ is an extension of $\psi\colon A\times A\to A$. 
Note that $\codim_{Z\times Z}\Sigma\geq 2$. 
We consider $g^*\mathcal O_{\mathbb P^N}(1)$. 
This line bundle can be extended 
to a line bundle $\mathcal L$ on $Z\times Z$. 
Moreover, we can see 
\begin{equation*} 
l_i:=g^*X_i \in 
H^0(Z\times 
Z, \mathcal L)
\end{equation*} 
for $0\leq i\leq N$, where 
$[X_0: \cdots :X_N]$ are homogeneous coordinates of $\mathbb P^N$. 
Therefore, we obtain a rational map 
$h\colon  Z\times Z\dashrightarrow Z$, which is 
given by the linear system spanned by $\{l_0, \cdots, l_N\}$ and 
is an extension of $g$. 
Thus, the multiplication 
\begin{equation*} 
\psi\colon  A\times A\to A
\end{equation*}  
is algebraic since $\psi=h|_{A\times A}$. 
Let 
\begin{equation*}
\iota\colon A\to A
\end{equation*} 
be the inverse. We can easily see that $\iota$ extends to 
\begin{equation*}
Z\setminus \Delta_{\mathrm{sing}}\to Z\hookrightarrow \mathbb P^N, 
\end{equation*} 
where $\Delta_{\mathrm{sing}}$ is the singular locus of $\Delta$. 
Hence, as in the case of $g$, 
we obtain a birational map $Z\dashrightarrow Z$, 
which is an extension of 
$\iota\colon A\to A$. Thus, $\iota$ is algebraic. 
This means that $A=\widetilde {\mathcal A}_V$ is 
an algebraic group. So, $\widetilde {\mathcal A}_V$ is a quasi-abelian 
variety. Note that the short exact sequence \eqref{p-eq3.1} is nothing but 
the Chevalley decomposition. 
\end{proof}

\begin{lem}\label{p-lem3.9}
Let $\omega$ be an element of $H^0(X, \Omega^1_X(\log D))$. 
We fix a point $0\in V$. 
Then we have a multivalued holomorphic 
function 
\begin{equation*} 
\int _0^p \omega
\end{equation*}  
on $V$. 
For a point $p\in V$, 
we can define $\alpha_V\colon  V\to \widetilde {\mathcal A}_V$ 
by  
\begin{equation*} 
\alpha_V(p)=\left ( \int _0^p \omega_1, \cdots, 
\int _0^p \omega_q, \int _0 ^p \varphi_1, \cdots, \int _0^p \varphi_d\right )
\in \widetilde {\mathcal A}_V. 
\end{equation*} 
This map is independent of the choice of the path from $0$ to 
$p$ in $V$. 
Thus we get a quasi-Albanese map: 
\begin{equation*} 
\alpha_V\colon  V\to \widetilde {\mathcal A}_V. 
\end{equation*} 
It is a holomorphic map. 
\end{lem}
\begin{proof}
Let $\gamma$ be a $2$-cycle 
on $V$. 
Then 
\begin{equation*} 
\int _{\partial \gamma}\omega=\int _{\gamma} d\omega=0
\end{equation*}  
for every $\omega\in H^0(X, \Omega^1_X(\log D))$. 
This is because $\omega$ is $d$-closed by Deligne (see \cite{deligne}). 
Therefore, $\alpha_V$ is well-defined. 
\end{proof}
\begin{lem}[{see \cite[Proposition 3]{iitaka1}}]\label{p-lem3.10} 
The map 
$\alpha_V$ in Lemma \ref{p-lem3.9} is algebraic. 
\end{lem}
\begin{proof}
Note that $A=\widetilde {\mathcal A}_V$ is a 
principal $\mathbb G^d_m$-bundle 
over $B=\mathcal A_X$ as a complex manifold. 
We consider the group homomorphism 
\begin{equation*}
\rho'\colon  \mathbb G^d_m\to \mathrm{PGL}(2, \mathbb C)\times 
\mathrm{PGL}(2, \mathbb C)\times \cdots \times 
\mathrm{PGL}(2, \mathbb C) 
\end{equation*} 
given by 
\begin{equation*} 
\rho'(\lambda_1, \cdots, \lambda_d)=\begin{pmatrix}
1&0\\0&\lambda_1\end{pmatrix} 
\times \begin{pmatrix}
1&0\\0&\lambda_2\end{pmatrix}
\times \cdots \times 
\begin{pmatrix}
1&0\\0&\lambda_d\end{pmatrix}.   
\end{equation*}  
Then we obtain a $\mathbb G^d_m$-equivariant embedding $Z'=A\times 
_{\rho'}(\mathbb P^1\times \cdots \times \mathbb P^1)$ of 
$A$ over $B$. 

\begin{claim}\label{p-3.10claim}
The holomorphic map 
\begin{equation*} 
\alpha_V\colon V\to \widetilde {\mathcal A}_V
\end{equation*}  
given in Lemma \ref{p-lem3.9} can be extended to 
a rational map
\begin{equation*} 
\beta_X\colon X\dashrightarrow Z'. 
\end{equation*} 
\end{claim}
\begin{proof}[Proof of Claim] 
We note that 
it is sufficient to prove that there exists a meromorphic extension 
$\beta_X$ of $\alpha_V$ since $X$ and $Z'$ are smooth 
complete algebraic varieties. 
Let $p$ be a point of $D\subset X$. 
Let $(z_1, \cdots, z_n)$ be a local holomorphic 
coordinate system of $X$ at $p$ such that 
$D$ is defined by $z_1\cdots z_r=0$. 
In this case, we can write 
\begin{equation*} 
\varphi_i=\sum _{b=1}^r \alpha_{ib} \frac{dz_b}{z_b}+\widetilde \varphi_i 
\end{equation*} 
where 
$\alpha_{ib}\in \mathbb C$ and 
$\widetilde \varphi_i$ is a holomorphic   
$1$-form for every $i$ around $p$ 
(see the proof of Lemma \ref{p-lem3.6}). 
Let $\delta_a$ be a circle around $D_a=(z_a=0)$ near $p$. 
Then $\iota_*\delta_a=0$. 
Therefore, we have 
\begin{equation*} 
\delta_a=\sum _j m_{ja}\eta_j+\widetilde \delta_a, 
\end{equation*}  
where $m_{ja}\in \mathbb Z$ and $\widetilde \delta_a$ is a torsion 
element. 
Thus we have 
\begin{equation*} 
\alpha_{ia}=\frac{1}{2\pi\sqrt{-1}}\int _{\delta_a} \varphi_i 
=\frac{1}{2\pi\sqrt{-1}}\sum _j m_{ja} 
\int _{\eta_j} \varphi_i=\frac{m_{ia}}{2\pi\sqrt{-1}}. 
\end{equation*}  
Without loss of generality, we may assume that $0\in V$ is near $p$. 
For a point $p'\in V$ near $p$, we have 
\begin{equation}\label{p-eq3.2}
\begin{split}
&\exp \left(2\pi\sqrt{-1}\int_{0}^{p'}
\varphi_i\right)\\ 
&=c\exp \left( \sum _b m_{ib}\log z_b(p')\right) \cdot 
\exp \left(2\pi\sqrt{-1}\int_{0}^{p'}\widetilde \varphi_i\right)
\\&=c\prod _b z_b(p')^{m_{ib}}\cdot 
\exp \left(2\pi\sqrt{-1}\int _{0}^{p'}\widetilde \varphi_i\right)
\end{split}
\end{equation}
for some constant $c$. 
We consider the following commutative diagram: 
\begin{equation*} 
\xymatrix{
V\ar[r]^{\alpha_V}\ar@{^{(}->}[d] & \widetilde {\mathcal A}_V\ar@{^{(}->}[d]\\
X\ar@{-->}[r]\ar[d]& Z'\ar[d]^{\pi_{Z'}}\\
B \ar@{=}[r]&B
}
\end{equation*} 
Note that $X\to B$  is nothing but the Albanese map of $X$. 
Let $U$ be a small open set of $B$ in the classical topology. 
Then 
\begin{equation*}
\pi^{-1}(U)\simeq U\times \mathbb C^\times 
\times \cdots \times \mathbb C^{\times}, 
\end{equation*}  
where $\pi\colon \widetilde {\mathcal A}_V\to B=\mathcal A_X$, 
and 
\begin{equation*}
\pi_{Z'}^{-1}(U)
\simeq U\times \mathbb P^1\times 
\cdots \times \mathbb P^1
\end{equation*} 
over $U$. 
Over $U$, it is easy to see that $\alpha_V$ can be extended to 
a meromorphic map $X\dashrightarrow Z'$ in the sense of Remmert 
by \eqref{p-eq3.2} 
(see \cite[Chapter 10, \S6, 3.~Graph of 
a Finite System of Meromorphic Functions]{grauert-remmert}). 
For the definition of meromorphic mappings in the 
sense of Remmert, see, for example, \cite[Definition 2.2]{ueno}. 
Therefore, $\alpha_V$ can be extended to a meromorphic map 
$\beta_X$ from 
$X$ to $Z'$ in the sense of Remmert. 
By Serre's GAGA principle (see, for example, 
\cite[Expos\'e XII]{sga1}), 
a meromorphic map 
$\beta_X\colon X\dashrightarrow Z'$ is 
a rational map between smooth complete algebraic varieties. 
\end{proof}
Thus we obtain that $\alpha_V$ in Lemma \ref{p-lem3.9} is algebraic. 
\end{proof}

\begin{lem}\label{p-lem3.11}
We have that 
\begin{equation*}
(\alpha_V)_*\colon 
H_1(V, \mathbb Z)\to H_1(\widetilde {\mathcal A}_V, \mathbb Z)
\end{equation*}  
is surjective. Moreover, we have 
\begin{equation*}
\mathrm{Ker}(\alpha_V)_*=H_1(V, \mathbb Z)_{\mathrm{tor}}, 
\end{equation*} 
where $H_1(V, \mathbb Z)_{\mathrm{tor}}$ is the torsion part of 
$H_1(V, \mathbb Z)$. 
\end{lem}
\begin{proof}
Let $H_1(V, \mathbb Z)_{\mathrm{free}}$ be the 
free part of $H_1(V, \mathbb Z)$. 
Note that $\widetilde {\mathcal A}_V=\mathbb C^{\overline q}/L$ by 
construction, where 
\begin{equation*}
\mathbb C^{\overline q} =
(H^0(X, \Omega^1_X(\log D)))^*
\end{equation*} 
and $L$ is an embedding 
of $H_1(V, \mathbb Z)_{\mathrm{free}}$ into $(H^0(X, \Omega^1_X(\log D)))^*$. 
On the other hand, 
\begin{equation*}
H_1(\widetilde {\mathcal A}_V, \mathbb Z)=\pi_1(\widetilde {\mathcal A}_V)=L 
\end{equation*}  
by construction. 
By the construction of the lattice $L$ 
and the quasi-Albanese map $\alpha_V\colon V\to 
\widetilde {\mathcal A}_V$, 
it is obvious that 
\begin{equation*} 
(\alpha_V)_*\colon 
H_1(V, \mathbb Z)\to H_1(\widetilde {\mathcal A}_V, \mathbb Z)
\end{equation*}  
is surjective and that 
\begin{equation*}
\mathrm{Ker}(\alpha_V)_*=H_1(V, \mathbb Z)_{\mathrm{tor}}. 
\end{equation*} 
This is the desired property. 
\end{proof}

\begin{lem}\label{p-lem3.12}
We have that 
\begin{equation*} 
{\alpha}_V^*\colon  T_1(\widetilde {\mathcal A}_V)\to 
T_1(V)
\end{equation*}  
is an isomorphism. 
\end{lem}
\begin{proof} 
By Lemma \ref{p-lem3.11}, 
\begin{equation*} 
(\alpha_V)_*\colon 
H_1(V, \mathbb Q)\to H_1(\widetilde {\mathcal A}_V, \mathbb Q)
\end{equation*}  
is an isomorphism. 
Therefore, 
we obtain 
\begin{equation*} 
(\alpha_V)^*\colon 
H^1(\widetilde {\mathcal A}_V, \mathbb Q)\to H^1(V, \mathbb Q) 
\end{equation*}  
is also an isomorphism. 
Moreover, it is an isomorphism of mixed 
Hodge structures (see \cite{deligne}). 
Therefore, we have an isomorphism 
\begin{equation*} 
{\alpha}_V^*\colon  T_1(\widetilde {\mathcal A}_V)\to T_1(V) 
\end{equation*}  
by Deligne (see \cite{deligne}). 
\end{proof}

The following lemma is useful and important. 

\begin{lem}\label{p-lem3.13}
Let $W$ be a quasi-abelian variety. 
Then the quasi-Albanese map 
\begin{equation*} 
\alpha_W\colon  W\to \widetilde {\mathcal A}_W
\end{equation*}  
is an isomorphism. 
\end{lem}
\begin{proof} 
By translation, we may assume that 
$\alpha_W(0)=0$. 
By Lemma \ref{p-lem3.11}, 
\begin{equation}\label{p-eq3.3}
(\alpha_W)_*\colon 
H_1(W, \mathbb Z)\to H_1(\widetilde {\mathcal A}_W, \mathbb Z)
\end{equation}
is an isomorphism. 
By Lemma \ref{p-lem3.12}, 
\begin{equation}\label{p-eq3.4}
{\alpha}_W^*\colon  T_1(\widetilde {\mathcal A}_W)\to 
T_1(W)
\end{equation}
is an isomorphism. 
Then $\alpha_W$ induces an isomorphism 
of complex vector spaces 
\begin{equation*}
(\alpha_W)_*\colon  T_{W, 0}\to T_{\widetilde {\mathcal A}_W, 0}, 
\end{equation*} 
where $T_{W, 0}$ is the tangent space of $W$ at $0$ and 
$T_{\widetilde {\mathcal A}_W, 0}$ is the tangent space of $\widetilde 
{\mathcal A}_W$ at $0$. 
By considering the exponential maps, 
we can recover $\alpha_W$ by 
\eqref{p-eq3.3} and \eqref{p-eq3.4}. 
By the isomorphisms in \eqref{p-eq3.3} and \eqref{p-eq3.4}, 
$\alpha_W$ is an isomorphism of complex Lie groups. 
Note that $\alpha_W$ is algebraic. 
Therefore, $\alpha_W$ is an isomorphism 
between 
smooth algebraic varieties. 
\end{proof}

\begin{lem}\label{p-lem3.14}
Let $f\colon V\to T$ be a morphism to a quasi-abelian variety $T$. 
Then there exists a unique algebraic morphism 
$\widetilde f\colon  \widetilde {\mathcal A}_V\to T$ such that 
$f=\widetilde f\circ \alpha_V$ 
\begin{equation*} 
\xymatrix{
V\ar[r]^f\ar[d]_{\alpha_V}&T\\
\widetilde {\mathcal A}_V\ar[ur]_{\widetilde f}&
}
\end{equation*} 
where 
$\alpha_V\colon 
V\to \widetilde {\mathcal A}_V$ is a quasi-Albanese map of 
$V$. 
\end{lem}

\begin{proof} 
We take a point $0\in V$. 
By translations, we may assume that 
$\alpha_V(0)=0$ and $f(0)=0$. 
Let $\{u_1, \cdots, u_k\}$ be a basis of $T_1(T)$. 
We may assume that 
\begin{equation*}
f^*u_1, \cdots, f^*u_l
\end{equation*}  
are linearly independent, where 
\begin{equation*} 
l=\dim _\mathbb C\langle f^*u_1, \cdots, f^*u_k\rangle. 
\end{equation*}  
We take $v_1, \cdots, v_m\in T_1(V)$ such that 
\begin{equation*} 
\{v_1, \cdots, v_m, f^*u_1, \cdots, f^*u_l\}
\end{equation*}  
is a basis of $T_1(V)$. 
Since 
$f_*\colon 
H_1(V, \mathbb Z)\to H_1(T, \mathbb Z)$, 
by using the basis 
$\{v_1, \cdots, v_m, f^*u_1, \cdots, f^*u_l\}$ of $T_1(V)$, 
we can easily construct a holomorphic map 
\begin{equation*} 
\xymatrix{
\widetilde f\colon  \widetilde {\mathcal A}_V\ar[r]& 
\widetilde {\mathcal A}_T
\ar[r]^{\sim}_{\alpha^{-1}_T}&
T
}
\end{equation*}  
(see Lemma \ref{p-lem3.13}) satisfying 
$f=\widetilde f\circ \alpha_V$. 
Therefore, there is a commutative diagram: 
\begin{equation*} 
\xymatrix{
T_1(V)&T_1(T)\ar[l]_{f^*}\ar[ld]^{\widetilde f^*}\\
T_1(\widetilde {\mathcal A}_V)\ar[u]^{\alpha^*_V}&
}
\end{equation*} 
which 
determines $\widetilde f^*$ uniquely. 
This is because 
$\alpha^*_V$ is an isomorphism (see Lemma \ref{p-lem3.12}). 
As in the proof of Lemma \ref{p-lem3.13}, 
by considering the exponential maps, 
we see that $\widetilde f$ can be uniquely 
recovered by $\widetilde f^*$. 
Thus, $\widetilde f$ is unique. 
Therefore, all we have to do is to prove that $\widetilde f$ is algebraic. 
It is sufficient to prove that the graph 
\begin{equation*}
\Gamma =\{(x, {\widetilde f}(x))\, |\, x\in \widetilde {\mathcal A}_V\}
\subset \widetilde {\mathcal A}_V\times T
\end{equation*}  
is an algebraic variety. 
We consider the map 
\begin{equation*} 
\alpha_n\colon  V^{2n}=V\times \cdots \times V\to \widetilde {\mathcal A}_V
\end{equation*}  
given by 
\begin{equation*} 
\alpha_n(z_1, \cdots, z_{2n})=\alpha_V(z_1)+\cdots 
+\alpha_V(z_n)-\alpha_V(z_{n+1})-\cdots -\alpha_V(z_{2n}). 
\end{equation*}  
We put 
\begin{equation*} 
F_n =\overline {\mathrm{Im}\alpha_n}, 
\end{equation*}  
that is, the Zariski closure of $\mathrm{Im}\alpha_n$. 
Then $F_n$ is an irreducible algebraic subvariety of $\widetilde 
{\mathcal A}_V$ for every $n$ such that 
\begin{equation*} 
F_1\subset F_2\subset \cdots 
\subset F_k \subset \cdots. 
\end{equation*}  
Therefore, there is a positive integer $n_0$ such that 
\begin{equation*} 
F_{n_0}=F_{n_0+1}=\cdots . 
\end{equation*}  
Note that $F_{n_0}$ is a quasi-abelian 
subvariety of $\widetilde {\mathcal A}_V$ because it is closed under the 
group law of $\widetilde {\mathcal A}_V$. 
Moreover, by the universality of $\widetilde {\mathcal A}_V$ 
proved above, $F_{n_0}$ is not contained 
in a quasi-abelian proper subvariety of $\widetilde {\mathcal A}_V$. 
This implies that $F_{n_0}=\widetilde {\mathcal A}_V$. 
Note that 
$\widetilde f$ is a homomorphism 
of complex Lie groups. 
We consider the following commutative diagram: 
\begin{equation*} 
\xymatrix{
V^{2n_0}\ar[d]_{\alpha_{n_0}}\ar[r]^{f_{n_0}}&T\\
\widetilde {\mathcal A}_V\ar[ur]_{\widetilde f}&
}
\end{equation*} 
where 
\begin{equation*} 
f_{n_0}(z_1, \cdots, z_{2n_0})=f(z_1)+\cdots 
+f(z_{n_0})-f(z_{n_0+1})-\cdots -f(z_{2n_0}). 
\end{equation*}  
We consider the Zariski closure of 
\begin{equation*} 
\{(\alpha_{n_0}(x), f_{n_0}(x))\, |\, x\in V^{2n_0}\}\subset 
\Gamma \subset \widetilde {\mathcal A}_V
\times T. 
\end{equation*} 
Then it is an algebraic subvariety 
of $\widetilde {\mathcal A}_V\times T$ and 
coincides with the graph $\Gamma$. 
This implies that $\widetilde f$ 
is algebraic. 
\end{proof}

\begin{lem}\label{p-lem3.15}
Let $f\colon V_1\to V_2$ be a morphism between smooth algebraic varieties. 
Then $f$ induces an algebraic 
morphism $f_*\colon  \widetilde {\mathcal A}_{V_1}\to 
\widetilde {\mathcal A}_{V_2}$ which satisfies the following commutative 
diagram. 
\begin{equation*} 
\xymatrix{
V_1\ar[r]^f\ar[d]_{\alpha_{V_1}}
&V_2\ar[d]^{\alpha_{V_2}}\\
\widetilde {\mathcal A}_{V_1}\ar[r]_{f_*}&\widetilde {\mathcal A}_{V_2}
}
\end{equation*} 
Moreover, $f_*$ is unique. 
\end{lem}
\begin{proof} 
It is almost obvious by Lemma \ref{p-lem3.14}. 
We apply Lemma \ref{p-lem3.14} to the map 
$\alpha_{V_2}\circ f\colon 
V_1\to \widetilde {\mathcal A}_{V_2}$. 
Then we obtain the desired map $f_*\colon  \widetilde {\mathcal A}_{V_1}\to 
\widetilde {\mathcal A}_{V_2}$ uniquely. 
\end{proof}
\end{say}

We summarize: 

\begin{thm}[Iitaka's quasi-Albanese varieties and maps]\label{p-thm3.16}
Let $V$ be a smooth algebraic variety. 
Then there exists a 
quasi-abelian variety $\widetilde {\mathcal A}_V$ and a morphism 
$\alpha_V\colon  V\to \widetilde {\mathcal A}_V$ with the following property: 
\begin{quote}
for any quasi-abelian variety $T$ and any morphism $f\colon V\to T$, 
there exists a unique morphism 
$\widetilde f\colon  \widetilde{\mathcal A}_V\to T$ such 
that $\widetilde f\circ \alpha_V=f$. 
\begin{equation*} 
\xymatrix{V\ar[r]^f\ar[d]_{\alpha_V}& T\\
\widetilde {\mathcal A}_V\ar[ru]_{\widetilde f}&
}
\end{equation*} 
\end{quote}
The quasi-abelian variety $\widetilde {\mathcal A}_V$, 
determined up to isomorphism by this 
condition, is called the {\em{quasi-Albanese variety}} of $V$. 
The map $\alpha_V\colon  V\to \widetilde {\mathcal A}_V$ 
is called the {\em{quasi-Albanese map}} of $V$. 
By the construction of $\widetilde {\mathcal A}_V$, 
$\widetilde {\mathcal A}_V$ is nothing but 
the Albanese variety of $V$ when $V$ is complete. 
\end{thm}

Anyway, Theorem \ref{p-thm3.16} is a generalization of 
the theory of Albanese maps and varieties 
for non-compact smooth 
complex algebraic varieties. 
We close this section with an easy corollary of Theorem \ref{p-thm3.16}. 

\begin{cor}[{cf.~Remark \ref{p-rem2.12}}]\label{p-cor3.17}
Let $f\colon \mathbb A^1\to G$ be an algebraic morphism from $\mathbb A^1$ to 
a quasi-abelian variety $G$. 
Then $f(\mathbb A^1)$ is a point. 
\end{cor}
\begin{proof}
Note that $T_1(\mathbb A^1)=0$. 
Thus the quasi-Albanese 
variety $\widetilde {\mathcal A}_{\mathbb A^1}$ is a point. 
Since $f$ factors through 
$\widetilde {\mathcal A}_{\mathbb A^1}$ by Theorem \ref{p-thm3.16}, 
$f(\mathbb A^1)$ is a point. 
\end{proof}

\section{Basic properties of quasi-abelian varieties}\label{p-sec4} 

In this section, we collect some basic properties of 
quasi-abelian varieties for the reader's convenience. 
We will use them in the proof of Theorem \ref{p-thm1.2}. 

\begin{say}\label{p-say4.1} 
Let $G$ be a quasi-abelian variety and let 
\begin{align}\label{p-eq4.1}
0\longrightarrow \mathcal G\longrightarrow G\longrightarrow \mathcal A
\longrightarrow 0
\end{align}
be the Chevalley decomposition such that 
$\mathcal G=\mathbb G^d_m$. 
We put $\dim \mathcal A=q$ and $n=\dim G=q+d$. 
Then there exists a $(2q+d)\times n$ matrix $M$ 
with 
\begin{equation*} 
M=\begin{pmatrix}
P &Q \\
0 &I_d
\end{pmatrix}
\end{equation*}  
where $P$ is a 
$2q\times q$ matrix 
and $I_d$ is the $d\times d$ unit matrix. 
The lattice spanned by the row vectors of $M$ 
(resp.~$P$) is denoted by $L$ (resp.~$L_1$). 
Then we have the short exact sequence of 
complex Lie groups:  
\begin{align}\label{p-eq4.2}
0\longrightarrow \mathbb G^d_m 
\longrightarrow \mathbb C^n/L\longrightarrow 
\mathbb C^q/L_1\longrightarrow 0. 
\end{align}
Note that 
\begin{equation*} 
\mathbb C^d/\mathbb Z^d\simeq \mathbb G^d_m
\end{equation*}  
by 
\begin{equation*}
(z_{q+1}, \cdots, z_n)\mapsto (\exp 2\pi\sqrt{-1} z_{q+1}, \cdots, \exp 
2\pi\sqrt{-1}z_n), 
\end{equation*}  
where $(z_1, \cdots, z_n)$ is the standard 
coordinate system of $\mathbb C^n$. 
By the descriptions in Section \ref{p-sec3}, 
the short exact sequence of complex Lie groups \eqref{p-eq4.2} 
is isomorphic to the short exact sequence \eqref{p-eq4.1}. 
The description $\mathbb C^n/L$ for $G$ is useful for 
various computations in the following theorems.  
\end{say}

We will repeatedly use the following theorem implicitly. 

\begin{thm}\label{p-thm4.2}
Let $G$ be a 
quasi-abelian variety. 
Assume that $\pi\colon G'\to G$ is a finite \'etale morphism 
from a variety $G'$. 
Then $G'$ is a quasi-abelian variety. 
\end{thm}

\begin{proof}
We use the notation in \ref{p-say4.1}. 
By \ref{p-say4.1}, 
$G=\mathbb C^n/L$. 
Then we have a 
sublattice $L'$ of $L$ such that 
$[L: L']<\infty$ and that $G'=\mathbb C^n/L'$. 
By a translation of $G$, we may assume that 
$\pi(0)=0$. 
Then we can easily construct a commutative diagram of 
complex Lie groups: 
\begin{equation*} 
\xymatrix{
0\ar[r]&\mathbb G^d_m\ar[r]\ar[d]^{\pi_2}&\mathbb C^n/L'
\ar[r]\ar[d]^{\pi}&\mathbb C^q/ 
L'_1 \ar[r]\ar[d]^{\pi_1}& 0\\
0\ar[r]&\mathbb G^d_m\ar[r]&\mathbb C^n/L\ar[r]&\mathbb C^q/ L_1 \ar[r]& 0
}
\end{equation*}  
such that $\pi$, $\pi_1$, and $\pi_2$ are finite. 
Since $\pi_1$ is finite, $\mathbb C^q/L'_1$ is an abelian variety. 
Note that $G'=\mathbb C^n/L'$ is a principal $\mathbb G^d_m$-bundle 
over $\mathbb C^q/L'_1$ as a complex 
manifold. By the proof of Lemma \ref{p-lem3.8}, 
the group law of $G'=\mathbb C^n/L'$ as a complex Lie group 
is algebraic. This means that 
$G'$ is a quasi-abelian variety and that $\pi\colon G'\to G$ is a 
group homomorphism between quasi-abelian varieties. 
\end{proof}

\begin{thm}[{\cite[10.]{iitaka1}}]\label{p-thm4.3}
Let $G$ be a quasi-abelian variety. 
Then we have 
$\overline \kappa (G)=0$ and $\overline q(G)=\dim G$. 
\end{thm}
\begin{proof}
Note that $G$ is a principal $\mathbb G^d_m$-bundle 
over an abelian variety $\mathcal A$ as a complex manifold. 
As in the proof of Lemma \ref{p-lem3.8}, 
we have a $\mathbb P^d$-bundle $\overline G$ over $\mathcal A$ such that 
$\overline G$ is a $\mathbb G^d_m$-equivariant embedding 
of $G$ over $\mathcal A$. 
We put $D=\overline G-G$. 
Then $D$ is a simple normal crossing divisor on $\overline G$. 
We can easily check that 
\begin{equation*} 
\Omega^1_{\overline G}(\log D)\simeq \oplus \mathcal O_{\overline G}. 
\end{equation*}  
More explicitly, 
$\Omega^1_{\overline G}(\log D)$ is isomorphic to 
$\oplus _{i=1}^n \mathcal O_{\overline G} d z_i$ in the notation of 
\ref{p-say4.1}. 
Therefore, we obtain that $\overline q(G)=\dim G$ and 
$K_{\overline G}+D\sim 0$. 
In particular, we have $\overline \kappa (G)=\kappa (\overline G, 
K_{\overline G}+D)=0$. 
\end{proof}

\begin{thm}[{cf.~\cite[Theorem 4.1]{iitaka1}}]\label{p-thm4.4}
Let $G$ be a quasi-abelian variety. 
Let $W$ be a closed subvariety of $G$. 
Then $\overline \kappa (W)\geq 0$. 
Moreover, $\overline \kappa (W)=0$ if and 
only if 
$W$ is a translation of a 
quasi-abelian subvariety of $G$.   
\end{thm}
\begin{proof}
We take a general point $p\in W$, around which we take a system of 
local analytic coordinates 
$(\zeta_1, \cdots, \zeta_n)$ such that 
\begin{equation*} 
W=(\zeta_{r+1}=\cdots =\zeta_n=0). 
\end{equation*}  
Let $\pi\colon \mathbb C^n\to G$ be the universal cover. 
We take $q\in \pi^{-1}(p)$ and assume that 
$z_1(q)=\cdots =z_n(q)=0$, where 
$(z_1, \cdots, z_n)$ is a system of 
global coordinates of $\mathbb C^n$. 
Note that $(\zeta_1, \cdots, \zeta_n)$ can be regarded as a system of 
local analytic coordinates around $q$. 
By taking a suitable linear transformation of $\mathbb C^n$, 
we have 
\begin{equation*} 
\zeta_j=z_j-\varphi_j(z_1, \cdots, z_n), 
\end{equation*}  
where $\varphi_j(0)=0$ and 
\begin{equation*} 
\frac{\partial \varphi_j}{\partial z_k}(0)=0
\end{equation*}  
for every $j$ and $k$ around $q$. 
The $dz_j$ defines a logarithmic $1$-form on $G$, that is, 
$dz_j\in T_1(G)$ for every $j$ (see the proof of Theorem \ref{p-thm4.3}). 
Let $f\colon V\to W$ be a resolution and let $\overline V$ be a smooth 
projective variety such that $\Delta=\overline V-V$ is a simple normal crossing 
divisor on $\overline V$. Without loss of generality, we may assume that 
$f$ is an isomorphism over a neighborhood of $p$.  
Thus, we see that $f^*(dz_j|_W)$ is an element of $T_1(W)$. 
Since $d\zeta_1, \cdots, d\zeta_r$ are linearly independent 
holomorphic 
$1$-forms on $W$ around $p$, 
$f^*(dz_1|_W), \cdots, f^*(dz_r|_W)$ are also 
linearly independent. Thus we have 
\begin{equation*} 
0\ne f^*(dz_1\wedge \cdots \wedge dz_r)|_W\in H^0(\overline V, 
\mathcal O_{\overline V}(K_{\overline V}+\Delta)). 
\end{equation*}  
This means that $\overline \kappa (W)\geq 0$. 
For $r+1\leq j\leq n$, we have 
\begin{equation*} 
dz_j|_{\pi^{-1}(W)}-
\sum _{k=1}^n \left.\frac{\partial \varphi_j}{\partial z_k}\right|_
{\pi^{-1}(W)} \cdot 
dz_k |_{\pi^{-1}(W)}=0 
\end{equation*} 
around $q$. 
Therefore, we obtain 
\begin{equation*} 
\sum _{k=r+1}^n\left(\delta_{jk}-
\left.\frac{\partial 
\varphi_j}{\partial z_k}\right|_{\pi^{-1}(W)}\right)dz_k|_{\pi^{-1}(W)}
=\sum _{i=1}^r 
\left.\frac{\partial \varphi_j}{\partial z_i}\right|_{\pi^{-1}(W)}
\cdot dz_i |_{\pi^{-1}(W)} 
\end{equation*}  
in a neighborhood of $q$. 
Thus, for $r+1\leq j\leq n$, 
we have 
\begin{equation*} 
dz_j|_W=\sum _{i=1}^rA_{ji}(\zeta_1, \cdots, \zeta_r)\cdot 
dz_i|_W 
\end{equation*}  
around $p$, where $A_{ji}$ is a holomorphic 
function for every $i$ and $j$ such that 
$A_{ji}(0)=0$. 
Note that 
\begin{equation*} 
f^*(dz_1|_W), \cdots, f^*(dz_r|_W) \in T_1(W), 
\end{equation*}  
which are linearly independent. 
Assume that $\overline \kappa (W)=0$. 
Then we have $\kappa (\overline V, K_{\overline V}+\Delta)=0$. 
We note that $f^*(dz_1\wedge \cdots \wedge dz_r)|_W$ is a nonzero 
element of $H^0(\overline V, \mathcal O_{\overline V}(K_{\overline V}+\Delta))$. 
Therefore, $H^0(\overline V, \mathcal O_{\overline V}(K_{\overline V}+\Delta))=
\mathbb C$ is spanned by $f^*(dz_1\wedge \cdots \wedge dz_r)|_W$. 
Thus, we obtain 
\begin{equation*} 
f^*(dz_2\wedge \cdots \wedge dz_r\wedge dz_j)|_W=\alpha_{1j} f^*
(dz_1\wedge \cdots \wedge dz_r)|_W, 
\end{equation*}  
where $\alpha_{1j}\in 
\mathbb C$ for every $j$.
On the other hand, 
\begin{equation*} 
f^*(dz_2\wedge \cdots \wedge dz_r\wedge dz_j)|_W=
\pm f^*\left(A_{j1}(dz_1\wedge \cdots \wedge dz_r)|_W\right)
\end{equation*}  
over a neighborhood of $p$. 
Hence we obtain $\pm f^*A_{j1}=\alpha _{1j}$ for every $j$. Note 
that $f$ is an isomorphism 
over a neighborhood of $p$. 
From this, $A_{j1}=0$ because $A_{j1}(0)=0$ for every $j$. 
By the same arguments, 
we get 
$A_{ji}=0$ for $1\leq i\leq r$ and every $j$. 
Thus, we obtain that 
\begin{equation*} 
dz_{r+1}|_W=\cdots =dz_n|_W=0 
\end{equation*}  
around $p$. 
This means that 
\begin{equation*} 
\pi^{-1}(W)\subset \{z_{r+1}=\cdots =z_n=0\} 
\end{equation*} 
near $q$. 
Note that $\{z_{r+1}=\cdots=z_n=0\}$ 
is of dimension $r$ and is irreducible. 
Thus 
\begin{equation*} 
\pi^{-1}(W)= \{z_{r+1}=\cdots =z_n=0\}.  
\end{equation*}  
Therefore, $W$ is a quasi-abelian subvariety of $G$. 
On the other hand, if $W$ is a translation of 
a quasi-abelian subvariety of $G$, then 
$\overline \kappa (W)=0$ by Theorem \ref{p-thm4.3}. 
\end{proof}

The following theorem is almost obvious by the description in 
\ref{p-say4.1}. 

\begin{thm}\label{p-thm4.5}
Let $G$ be a quasi-abelian variety. 
Then there are at most countably many quasi-abelian subvarieties 
of $G$. 
\end{thm}
\begin{proof}
Let $H$ be a quasi-abelian subvariety of $G$. 
Then we obtain 
\begin{equation*} 
\iota\colon  H=\mathbb C^{\dim H}/H_1(H, \mathbb Z)
\hookrightarrow 
\mathbb C^{\dim G}/H_1(G, \mathbb Z), 
\end{equation*}  
where $\iota$ is the natural inclusion. 
Anyway, $\iota$ is determined by the subgroup 
$\mathrm{Im}\iota_*$ of $H_1(G, \mathbb Z)$, 
where $\iota_*\colon  H_1(H, \mathbb Z)\to H_1(G, \mathbb Z)$. 
Therefore, there are at most countably many quasi-abelian subvarieties of $G$. 
\end{proof}

\section{Characterizations of abelian varieties 
and complex tori}\label{p-sec5}

In this section, we quickly recall an important property of 
the Albanese map of varieties of the Kodaira dimension zero for 
the reader's convenience. 
It is well known that Kawamata established the following theorem 
in \cite{kawamata-abelian}, which is his doctoral thesis. 

\begin{thm}[{see \cite[Theorem 1]
{kawamata-abelian}}]\label{p-thm5.1}
Let $X$ be a smooth projective variety with $\kappa (X)=0$. 
Then the Albanese map 
\begin{equation*} 
\alpha\colon X\to A
\end{equation*}  
is surjective and has connected fibers. 
\end{thm}

As an obvious corollary of Theorem \ref{p-thm5.1}, we obtain 
a birational characterization of abelian varieties. 

\begin{cor}\label{p-cor5.2}
Let $X$ be a smooth projective variety. 
Then $X$ is birationally equivalent to an abelian variety 
if and only if the Kodaira dimension 
$\kappa (X)=0$ and the irregularity $q(X)=\dim X$. 
\end{cor}

In this paper, we will use Theorem \ref{p-thm5.1} and Corollary \ref{p-cor5.2} 
for the proof of Theorem \ref{p-thm1.2}. 
For compact K\"ahler manifolds, we have: 

\begin{thm}[{see \cite[Theorem 24]
{kawamata-abelian}}]\label{p-thm5.3}
Let $X$ be a compact K\"ahler manifold with $\kappa (X)=0$. 
Then the Albanese map 
\begin{equation*} 
\alpha\colon X\to A
\end{equation*}  
is surjective and has connected fibers. 
\end{thm}

Therefore, we have: 

\begin{cor}\label{p-cor5.4}
Let $X$ be a compact K\"ahler manifold. 
Then $X$ is bimeromorphic to a complex torus if and 
only if the Kodaira dimension $\kappa (X)=0$ and 
the irregularity $q(X)=\dim X$. 
\end{cor}

Kawamata's original arguments in \cite{kawamata-abelian} 
heavily depend on the theory of variations of Hodge structure 
(see Section \ref{p-sec7} below). 
In \cite[Section 2]{ein-lazarsfeld},  
Ein and Lazarsfeld give a new proof of the above results. 
Their arguments are based 
on the generic vanishing theorem due to Green--Lazarsfeld. 
Anyway, the results in this section 
can be proved without using \cite{kawamata-abelian} 
now. Note that Theorem \ref{p-thm1.2} is a generalization of 
Theorem \ref{p-thm5.1}. We will give a 
detailed proof of Theorem \ref{p-thm1.2} in 
Section \ref{p-sec10} (see 
Theorem \ref{p-thm10.1}) following \cite{kawamata-abelian}. 
The author does not know any proofs of Theorem \ref{p-thm1.2} 
which are 
independent of Theorem \ref{p-thm6.1} below and only 
depend on the generic vanishing theorem due to Green--Lazarsfeld. 

\section{On subadditivity of 
the logarithmic Kodaira dimensions}\label{p-sec6}

In this section, we explain some known results on 
the subadditivity of the logarithmic Kodaira dimensions. 

\begin{thm}\label{p-thm6.1} 
Let $f\colon X\to Y$ be a dominant 
morphism between smooth varieties with 
irreducible general fibers. 
Assume that the logarithmic Kodaira 
dimension $\overline \kappa (Y)=\dim Y$. 
Then we have 
\begin{align*}
\overline \kappa (X)&=\overline \kappa (F)+\overline \kappa (Y)
\\&=\overline \kappa (F)
+\dim Y
\end{align*}
where $F$ is a sufficiently general fiber of $f\colon X\to Y$. 
\end{thm}

\begin{rem}\label{p-rem6.2}
Theorem \ref{p-thm6.1} is a 
generalization of \cite[Theorem 30]{kawamata-abelian}. 
In \cite{kawamata-abelian}, 
Kawamata claimed Theorem \ref{p-thm6.1} 
under the extra assumption 
that $\overline \kappa (X)\geq 0$. 
Theorem \ref{p-thm6.1} was first 
obtained by Maehara (see \cite[Corollary 2]{maehara}). 
Note that the arguments in \cite{kawamata-abelian} and \cite{maehara} 
heavily depend on \cite[Theorem 32]{kawamata-abelian}. 
Since the author has been unable to follow \cite[Theorem 32]{kawamata-abelian}, 
he gave a proof of Theorem \ref{p-thm6.1} which is 
independent of \cite[Theorem 32]{kawamata-abelian}. 
For the details, see \cite[Theorem 1.9]{fujino-weak} 
(see also Section \ref{p-sec7}). 
\end{rem}

Theorem \ref{p-thm6.3} is the 
main theorem of \cite{kawamata-master} 
(see \cite[Theorem 1]{kawamata-master}). 
For the proof, we recommend the reader 
to see \cite[Chapter 5]{iitaka-conjecture}. 

\begin{thm}\label{p-thm6.3}
Let $f\colon X\to Y$ be a dominant 
morphism between smooth varieties whose 
general fibers are irreducible curves. 
Then we have 
\begin{equation*} 
\overline \kappa (X)\geq \overline \kappa (F)+\overline \kappa (Y) 
\end{equation*} 
where $F$ is a general fiber of $f\colon X\to Y$. 
\end{thm}

Theorems \ref{p-thm6.1} and \ref{p-thm6.3} 
will play a crucial role in 
Sections \ref{p-sec9}, \ref{p-sec10}, and 
\ref{p-sec11}. 
In general, we have: 

\begin{conj}\label{p-conj6.4}
Let $f\colon X\to Y$ be a dominant 
morphism between smooth 
varieties whose general fibers are irreducible. 
Then we have 
\begin{equation*}
\overline \kappa (X)\geq \overline \kappa (F)+\overline \kappa (Y)
\end{equation*} 
where $F$ is a sufficiently general fiber of $f\colon X\to Y$. 
\end{conj}

By \cite{fujino-subadditivity} and \cite{fujino-corrigendum}, 
we see that Conjecture \ref{p-conj6.4} follows from the minimal model 
program and the abundance conjecture. 
For the details, see \cite{fujino-subadditivity}, 
\cite{fujino-corrigendum}, and \cite{fujino-mixed-omega}. 

\section{Remarks on semipositivity theorems}\label{p-sec7}

In this section, 
we make some comments on the semipositivity 
theorems in \cite{kawamata-abelian} for the reader's 
convenience. We recommend the reader to skip this section 
if he is only interested in Theorem \ref{p-thm1.2}. 
The arguments in Section \ref{p-sec8} are sufficient for the proof of 
Theorem \ref{p-thm1.2} and are more elementary. 
Let us recall Kawamata's famous result in \cite{kawamata-abelian}. 
It is one of the main ingredients of Kawamata's proof of 
Theorem \ref{p-thm5.1}. 

\begin{thm}[{\cite[Theorem 5=Main Lemma]{kawamata-abelian}}]\label{p-thm7.1}
Let $f\colon X\to Y$ be a surjective morphism between smooth projective 
varieties with connected fibers which 
satisfies the following conditions: 
\begin{itemize}
\item[(i)] There is a Zariski open dense subset $Y_0$ of $Y$ such that 
$\Sigma=Y-Y_0$ is a simple normal crossing divisor on $Y$. 
\item[(ii)] Put $X_0=f^{-1}(Y_0)$ and $f_0=f|_{X_0}$. 
Then f$_0$ is smooth. 
\item[(iii)] The local monodromies of $R^nf_{0*}\mathbb C_{X_0}$ around 
$\Sigma$ are unipotent, where $n=\dim X-\dim Y$. 
\end{itemize}
Then $f_*\mathcal O_X(K_{X/Y})$ is a locally free sheaf and 
semipositive, where $K_{X/Y}=K_X-f^*K_Y$. 
\end{thm}

\begin{rem}\label{p-rem7.2}
In \cite[\S4.~Semi-positivity (1)]{kawamata-abelian}, 
Kawamata proved that $f_*\mathcal O_X(K_{X/Y})$ coincides 
with the canonical extension of the bottom Hodge filtration $\mathcal F$. 
This part was generalized by Nakayama and Koll\'ar independently 
(see \cite[Theorem 1]{nakayama} and \cite[Theorem 2.6]{kollar}). 
They proved that $R^if_*\mathcal O_{X}(K_{X/Y})$ is locally free and 
can be characterized as the (upper) canonical extension of 
the bottom Hodge filtration of 
a suitable variation of Hodge structure. 
\end{rem}

\begin{rem}\label{p-rem7.3}
In \cite[\S4.~Semi-positivity (2)]{kawamata-abelian}, Kawamata proved that 
the canonical extension of 
the bottom Hodge filtration $\mathcal F$ is semipositive. 
This part is not so easy to follow. 
Kawamata's proof seems to be insufficient. 
Note that Kawamata could and did 
use only \cite{deligne}, \cite{griffiths}, and \cite{schmid} 
for the Hodge theory when \cite{kawamata-abelian} was written 
around 1980. 
Fortunately, \cite[Theorem 1.3]{fujino-fujisawa} and 
\cite[Theorem 3]{ffs} completely 
generalize \cite[\S4.~Semi-positivity (2)]{kawamata-abelian} 
for admissible variations of mixed Hodge structure and 
clarify Kawamata's proof simultaneously 
(see also \cite[Theorem 1.1 and Corollary 1.2]{fujino-fujisawa-mrl}). 
For Morihiko Saito's comments on Kawamata's arguments in 
\cite[\S4.~Semi-positivity (2)]{kawamata-abelian}, 
see \cite[4.6.~Remarks]{ffs}. 
\end{rem}

\begin{rem}\label{p-rem7.4} 
An approach to the semipositivity of $R^if_*\mathcal O_X(K_{X/Y})$ which 
dose not use \cite[\S4.~Semi-positivity (2)]{kawamata-abelian} 
can be found in \cite[Section 4]{fujino-remarks}. 
\end{rem}

Anyway, \cite[Theorem 5]{kawamata-abelian} is now clearly understood. 
Let us go to a mixed generalization of Theorem \ref{p-thm7.1}, 
which was used in the proof of Theorem \ref{p-thm6.1} in \cite{fujino-weak}. 
In \cite{fujino-higher}, we obtain:   

\begin{thm}[{see 
\cite[Theorems 3.1, 3.4, and 3.9]{fujino-higher}}]\label{p-thm7.5}
Let $f\colon X\to Y$ be a surjective morphism 
between smooth projective varieties and let $D$ be 
a simple normal crossing divisor on $X$ such that 
every stratum of $D$ is dominant onto $Y$. 
Let $\Sigma$ be a simple normal crossing divisor 
on $Y$. 
If $f$ is smooth and $D$ is relatively normal crossing 
over $Y_0=Y\setminus \Sigma$ and 
the local monodromies 
of $R^{n+i}f_{0*}\mathbb C_{X_0\setminus D_0}$ 
around $\Sigma$ are unipotent, where 
$X_0=f^{-1}(Y_0)$, $D_0=D|_{X_0}$, $f_0=f|_{X_0}$, and 
$n=\dim X-\dim Y$, 
then $R^if_*\mathcal O_X(K_{X/Y}+D)$ is locally free and 
semipositive. 
\end{thm}

Theorem \ref{p-thm7.5} is obviously a generalization of Theorem \ref{p-thm7.1}. 

\begin{rem}\label{p-rem7.6} 
In \cite{fujino-higher}, we characterize 
$R^if_*\mathcal O_X(K_{X/Y}+D)$ as the canonical 
extension of the bottom Hodge 
filtration of a suitable variation of mixed Hodge structure. 
The proof of the semipositivity of $R^if_*\mathcal O_X(K_{X/Y}+D)$ 
in \cite{fujino-higher} used \cite[\S4.~Semi-positivity (2)]{kawamata-abelian}. 
Now we can use \cite[Theorem 1.3]{fujino-fujisawa} 
or \cite[Theorem 3]{ffs} for the semipositivity of $R^if_*\mathcal O_X(K_{X/Y}
+D)$ in place of \cite[\S4.~Semi-positivity (2)]{kawamata-abelian} 
(see also \cite[Theorem 1.1 and Corollary 1.2]{fujino-fujisawa-mrl}). 
\end{rem}

\begin{rem}\label{p-rem7.7}  
As we pointed out in \cite[Remark 6.5]{fujino-weak}, 
Kawamata seems to misuse Schmid's nilpotent orbit theorem in 
\cite{kawamata-semi} and \cite{kawamata-hodge}. 
Therefore, we do not use the papers \cite{kawamata-semi} and 
\cite{kawamata-hodge}. 
Moreover, the main theorem of \cite{kawamata-semi} (see 
\cite[Theorem 1.1]{kawamata-semi}) is weaker than 
\cite[Theorem 1.1]{fujino-fujisawa}. 
\end{rem}

\begin{rem}\label{p-rem7.8} 
The main theorem in \cite{kawamata-semi} (see 
\cite[Theorem 1.1]{kawamata-semi}) 
does not cover Theorem \ref{p-thm7.5} nor 
\cite[Theorem 32]{kawamata-abelian}. 
We also note that 
\cite{kawamata-mixed-hodge} does not cover Theorem \ref{p-thm7.5} nor 
\cite[Theorem 32]
{kawamata-abelian}. In \cite{kawamata-mixed-hodge}, 
Kawamata treats {\em{well prepared fiber spaces}}. 
For the details, see \cite{kawamata-mixed-hodge}. 
The author did not find 
any proofs of \cite[Theorem 32]{kawamata-abelian} in the literature 
except the original one in \cite{kawamata-abelian}. 
\end{rem}

\section{Weak positivity theorems revisited}\label{p-sec8}

In this section, we explain how to avoid using the 
theory of variations of mixed Hodge structure for the proof of 
Theorem \ref{p-thm6.1}. Let us recall the definition of 
weakly positive sheaves. 
Note that the theory of weakly positive sheaves is due to Viehweg. 
Roughly speaking, Viehweg treated only the {\em{pure}} case. 
For the details of the 
{\em{mixed}} case, 
we recommend the reader to see \cite{fujino-weak}. 
For a slightly different approach, see \cite{fujino-mixed-omega}. 

\begin{defn}[Weak positivity]\label{p-def8.1}
Let $W$ be a smooth projective variety and let $\mathcal F$ be a torsion-free 
coherent sheaf on $W$. 
We call $\mathcal F$ {\em{weakly positive}}, 
if for every ample line bundle $\mathcal H$ on $W$ and every positive integer 
$\alpha$ there exists some positive integer $\beta$ such that 
$\widehat S^{\alpha\beta}(\mathcal F)\otimes \mathcal H^{\otimes \beta}$ is 
generically generated by global sections. 
This means that the natural map 
\begin{equation*} 
H^0(W, \widehat S^{\alpha\beta}(\mathcal F)\otimes \mathcal H^{\otimes \beta})
\otimes \mathcal O_W\to 
\widehat S^{\alpha\beta}(\mathcal F)\otimes \mathcal H^{\otimes 
\beta}
\end{equation*}  
is generically surjective. 
\end{defn}

\begin{rem}\label{p-rem8.2}
In Definition \ref{p-def8.1}, 
let $\widehat W$ be the largest Zariski open subset 
of $W$ such that 
$\mathcal F|_{\widehat W}$ is locally free. 
Then we put 
\begin{equation*} 
\widehat S^k(\mathcal F)=i_*S^k(i^*\mathcal F) 
\end{equation*} 
where $i\colon \widehat W\to W$ is the natural open immersion and 
$S^k$ denotes the $k$-th symmetric product. 
Note that 
$\codim_W(W\setminus \widehat W)\geq 2$ since $\mathcal F$ is 
torsion-free. 
\end{rem}

The following theorem, which is due to Viehweg, Campana, and 
others, is useful and is very important. 
 
\begin{thm}[{Twisted weak positivity, see \cite[Theorem 1.1]{fujino-weak}}]
\label{p-thm8.3} 
Let $X$ be a normal projective variety and let $\Delta$ be an effective 
$\mathbb Q$-divisor 
on $X$ such that 
$(X, \Delta)$ is log canonical. 
Let $f\colon X\to Y$ be a surjective 
morphism onto a smooth 
projective variety $Y$ with connected fibers. 
Assume that 
$k(K_X+\Delta)$ is Cartier. 
Then, for every positive integer 
$m$, 
\begin{equation*} 
f_*\mathcal O_X(mk(K_{X/Y}+\Delta))
\end{equation*} 
is weakly positive. 
\end{thm}

Once we establish Theorem \ref{p-thm8.3}, we can 
prove Theorem \ref{p-thm6.1} without any difficulties. 
For the details, see \cite[Section 10]{fujino-weak}. 
Theorem \ref{p-thm8.3} is 
sufficient for \cite[Sections 9 and 10]{fujino-weak}. 
A key ingredient of Theorem \ref{p-thm8.3} is the following result. 

\begin{thm}[{see \cite[Corollary 7.11]{fujino-weak}}]\label{p-thm8.4} 
Let $f\colon V\to W$ be a surjective 
morphism between smooth 
projective varieties. 
Let $D$ be a simple normal crossing divisor on $V$. 
Then 
\begin{equation*} 
f_*\mathcal O_V(K_{V/W}+D)
\end{equation*}  
is weakly positive. 
\end{thm}

By Theorem \ref{p-thm8.4}, 
the arguments in \cite[Section 8]{fujino-weak} 
work without any modifications and produce Theorem \ref{p-thm8.3}. 
We recommend the reader to see \cite[Section 8]{fujino-weak}. 
In \cite[Section 7]{fujino-weak}, we give a proof of Theorem \ref{p-thm8.4} 
based on the theory of variations of mixed Hodge structure 
(cf.~Theorem \ref{p-thm7.5}). 
Here, we give a more elementary proof 
based on the following easy observation. 

\begin{lem}\label{p-lem8.5}
Let $f\colon X\to Y$ be a 
surjective morphism 
from a smooth projective variety $X$ to a projective 
variety $Y$ and let $D$ be a simple normal crossing 
divisor on $X$. 
Let $\mathcal A$ be an ample 
line bundle on $Y$ such 
that $|\mathcal A|$ is free and let $\mathcal B$ be a line bundle 
on $Y$ such that $\mathcal A^{\otimes a}\otimes \mathcal B$ is nef for 
some positive integer $a$. 
Then 
\begin{equation*}
R^if_*\mathcal O_X(K_X+D)
\otimes \mathcal B\otimes \mathcal A^{\otimes m}
\end{equation*}  
is generated by global sections 
for every $i$ and every positive integer 
$m\geq \dim Y+1+a$.  
\end{lem}

\begin{proof}[Proof of Lemma \ref{p-lem8.5}] 
By \cite[Theorem 2.6]{fujino-higher} (see also \cite[Theorem 6.3 (ii)]
{fujino-fundamental}), 
we obtain that 
\begin{equation*} 
H^p(Y, R^if_*\mathcal O_X(K_X+D)\otimes \mathcal B\otimes \mathcal A^{\otimes a}
\otimes \mathcal A^{m-a-p})=0
\end{equation*}  
for $p>0$. By Castelnuovo--Mumford regularity, we see 
that 
$R^if_*\mathcal O_X(K_X+D)\otimes 
\mathcal B\otimes \mathcal A^{\otimes m}$ 
is generated by global sections for 
every $i$ and $m\geq \dim Y+1+a$. 
\end{proof}

Let us start the proof of Theorem \ref{p-thm8.4}. 

\begin{proof}[Proof of Theorem \ref{p-thm8.4}] 
In Step \ref{p-8.4-step1}, we reduce the problem to a simpler case. 
In Step \ref{p-8.4-step2}, we use Viehweg's clever trick and 
obtain the desired weak positivity. 

\begin{step}\label{p-8.4-step1}
By replacing $D$ with its horizontal part, we may assume that 
every irreducible component of $D$ is dominant 
onto $W$ (see \cite[Lemma 7.7]{fujino-weak}). 
If there is a log canonical center $C$ of $(V, D)$ such that 
$f(C)\subsetneq W$, then we take the blow-up 
$h\colon V'\to V$ along $C$. 
We put 
\begin{equation*}
K_{V'}+D'=h^*(K_V+D). 
\end{equation*}  
Then $D'$ is a simple normal crossing divisor on $V'$ and 
\begin{equation*} 
f_*\mathcal O_V(K_{V/W}+D)
\simeq (f\circ h)_*\mathcal O_{V'}(K_{V'/W}+D').
\end{equation*}  
Therefore, we can replace $(V, D)$ with $(V', D')$. 
Then we replace $D$ with its horizontal part (see 
\cite[Lemma 7.7]{fujino-weak}). 
By repeating this process finitely many times, 
we may assume that every stratum of $D$ is dominant onto $W$. 
Now we take a closed subset $\Sigma$ of $W$ such that 
$f$ is smooth over $W\setminus \Sigma$ and that $D$ is relatively 
normal crossing over $W\setminus \Sigma$. 
Let $g\colon W'\to W$ be a birational morphism from 
a smooth projective variety $W'$ such that 
$\Sigma'=g^{-1}(\Sigma)$ is a simple 
normal crossing divisor. By taking some suitable 
blow-ups of $V$ in $f^{-1}(\Sigma)$ and 
replacing $D$ with its strict transform, 
we may further assume the following 
conditions: 
\begin{itemize}
\item[(i)] $f'=g^{-1}\circ f\colon V\to W'$ 
is a morphism, 
\item[(ii)] $f'$ is smooth over $W'\setminus \Sigma'$ 
and $D$ is relatively 
normal crossing over $W'\setminus \Sigma'$, and 
\item[(iii)] every irreducible 
component of $D$ is dominant 
onto $W$ and $\Supp (f'^*\Sigma'+D)$ 
is a simple normal crossing divisor 
on $V$. 
\end{itemize}
\begin{equation*}
\xymatrix{
V\ar[d]_{f'} \ar[dr]^{f}& \\
   W' \ar[r]_{g} & W
} 
\end{equation*} 
Here we used Szab\'o's resolution lemma. 
We assume that 
$f'_*\mathcal O_V(K_{V/W'}+D)$ is weakly positive. 
Note that 
\begin{equation*} 
f'_*\mathcal O_V(K_{V/W}+D)\simeq f'_*\mathcal 
O_V(K_{V/W'}+D)\otimes \mathcal O_{W'}(E)
\end{equation*}  
where $E$ is a $g$-exceptional effective divisor such that $K_{W'}
=g^*K_W+E$.
Thus $f'_*\mathcal O_V(K_{V/W}+D)$ is weakly positive. 
We note that 
\begin{equation*}
g_*f'_*\mathcal O_{V}(K_{V/W}+D)\simeq f_*\mathcal O_V(K_{V/W}+D). 
\end{equation*}  
We can take an effective $g$-exceptional divisor $F$ on $W'$ such that 
$-F$ is $g$-ample. 
Let $H$ be an ample Cartier divisor 
on $W$. 
Then there exists a positive integer $k$ such that 
$kg^*H-F$ is ample. 
Let $\alpha$ be a positive integer. 
Since $f'_*\mathcal O_V(K_{V/W}+D)$ is weakly positive, 
\begin{equation*} 
\widehat{S}^{k\alpha\beta}(f'_*\mathcal O_V(K_{V/W}+D))\otimes 
\mathcal O_{W'}(\beta(kg^*H-F))
\end{equation*}  
is generically generated by 
global sections for some positive integer $\beta$. 
By taking $g_*$, 
\begin{equation*} 
\widehat {S}^{\alpha k \beta}(f_*\mathcal O_V(K_{V/W}+D))\otimes 
\mathcal O_{W}(k\beta H)
\end{equation*}  
is generically generated by global sections. 
This means that 
$f_*\mathcal O_V(K_{V/W}+D)$ is weakly positive. 
Therefore, all we have to do is to prove that 
$f'_*\mathcal O_V(K_{V/W'}+D)$ is weakly positive. 
\end{step}
\begin{step}\label{p-8.4-step2}
By replacing $W$ with $W'$, we may assume that 
$W'=W$. 
Note that $f_*\omega_{V/W}$ and 
$f_*(\omega_{V/W}\otimes \mathcal O_V(D))$ are 
locally free on $W_0=W\setminus \Sigma$. 
Let $s$ be an arbitrary positive integer. 
We take the $s$-fold fiber product 
\begin{equation*} 
V^s=V\times _WV\times _W\cdots \times _WV. 
\end{equation*}  
We put $f^s\colon V^s\to W$. 
Let $p_i\colon V^s\to V$ be the $i$-th projection for $1\leq i\leq s$. 
Let $W^\dag$ be a Zariski open set of $W$ such that 
$f$ is flat over $W^\dag$ and that 
$\codim_W(W\setminus W^\dag)\geq 2$. 
We may assume that 
$W_0\subset W^\dag\subset W$. 
We put $V^\dag=f^{-1}(W^\dag)$. 
We may further assume that $f_*\omega_{V^\dag/ W^\dag}$ and 
$f_*(\omega_{V^\dag/W^\dag}\otimes \mathcal O_{V^\dag}(D))$ are locally 
free. 
By the flat base change theorem (see, for example, 
\cite[Section 4]{mori} and \cite[Section 3.1]{iitaka-conjecture}), 
we obtain an isomorphism 
\begin{equation*} 
f^s_*\omega_{{V^\dag}^s/W^\dag}
\simeq \bigotimes _{i=1}^sf_*\omega_{V^\dag/W^\dag}, 
\end{equation*}  
where ${V^\dag}^s$ is the $s$-fold 
fiber product 
\begin{equation*} 
V^\dag\times _{W^\dag}\cdots \times _{W^\dag}V^\dag. 
\end{equation*} 
We put $D^s=\sum _{i=1}^sp_i^*D$. 
Then, by the same argument, 
we have an isomorphism 
\begin{equation}\label{p-eq8.1}
f^s_*(\omega_{{V^\dag}^s/W^\dag}\otimes \mathcal O_{{V^\dag}^s}(D^s))
\simeq 
\bigotimes_{i=1}^sf_*(\omega_{V^\dag/W^\dag}
\otimes \mathcal O_{V^\dag}(D)). 
\end{equation} 
Let $\pi\colon V^{(s)}\to V^s$ be a resolution such that 
$\pi$ is an isomorphism over $(f^s)^{-1}(W_0)$ with the 
following properties: 
\begin{itemize}
\item[(i)] $V^{(s)}$ is a smooth projective variety, 
\item[(ii)] $f^{(s)}=f^s\circ \pi\colon  V^{(s)}\to W$ is smooth over $W_0$, 
\item[(iii)] $D^{(s)}$ is a simple normal crossing divisor on $V^{(s)}$, 
\item[(iv)] $\Supp (D^{(s)}+(f^{(s)})^*\Sigma)$ is a simple normal crossing 
divisor on $V^{(s)}$, 
\item[(v)] every irreducible component of $D^{(s)}$ is dominant onto 
$W$, and 
\item[(vi)] $D^{(s)}$ coincides with $D^s$ over $W_0$. 
\end{itemize}
Note that ${V^\dag}^s$ is Gorenstein. 
We have 
\begin{equation*} 
\pi_*\mathcal O_{{V^\dag}^{(s)}}(K_{{V^\dag}^{(s)}})\subset 
\omega _{{V^\dag}^s}, 
\end{equation*}  
where ${V^\dag}^{(s)}=\pi^{-1}({V^\dag}^s)$. 
Therefore, we obtain 
\begin{equation}\label{p-eq8.2}
\pi_*\mathcal O_{{V^\dag}^{(s)}}(K_{{V^\dag}^{(s)}}+D^{(s)}-\pi^*D^s)\subset 
\omega _{{V^\dag}^s}
\end{equation} 
since $D^{(s)}-\pi^*D^s\leq 0$. 
Thus we have 
a natural inclusion 
\begin{align*}
f^{(s)}_*\mathcal O_{V^{(s)}}(K_{V^{(s)}/W}+D^{(s)})
\hookrightarrow 
\left(\bigotimes _{i=1}^sf_*\mathcal O_V(K_{V/W}+D)\right)^{**}
\end{align*}
which is an isomorphism over $W_0$ by \eqref{p-eq8.1} and \eqref{p-eq8.2}. 
Let $\mathcal H$ be an ample line bundle on $W$. 
Then 
\begin{equation*} 
f^{(s)}_*\mathcal O_{V^{(s)}}(K_{V^{(s)}/W}+D^{(s)})\otimes 
\mathcal H^{\otimes m}
\end{equation*} 
is generated by global sections for every positive integer $s$ 
and for every $m\geq b(\dim W+1)+a$, where 
$a$ is a positive integer such that 
$\mathcal O_W(-K_W)\otimes \mathcal H^{\otimes a}$ is nef and 
$b$ is a positive integer such that $|\mathcal H^{\otimes b}|$ is free by 
Lemma \ref{p-lem8.5}. 
Therefore, we obtain that 
\begin{equation*} 
\left(\bigotimes _{i=1}^sf_*\mathcal O_V(K_{V/W}+D)\right)^{**}
\otimes \mathcal H^{\otimes m}
\end{equation*}  
is generated by global sections over $W_0$, where 
$s$ and $m$ are as above. 
This means that 
\begin{equation*}
\widehat{S}^{\alpha \beta}(f_*\mathcal O_V(K_{V/W}+D))
\otimes \mathcal H^{\otimes \beta} 
\end{equation*}  
is generated by global sections over $W_0$ for 
every $\alpha \geq 1$ and $\beta\geq b(\dim W+1)+a$. 
Therefore, $f_*\mathcal O_V(K_{V/W}+D)$ is weakly positive. 
\end{step}
We complete the proof of Theorem \ref{p-thm8.4}. 
\end{proof}

\begin{rem}\label{p-rem8.6}
Note that $f_*\mathcal O_V(K_{V/W}+D)$ is locally free 
in Step \ref{p-8.4-step2} in the 
proof of Theorem \ref{p-thm8.4}. 
This is because $f_*\mathcal O_V(K_{V/W}+D)$ is the 
upper canonical extension of the bottom Hodge filtration of a suitable 
variation of mixed Hodge structure (cf.~Theorem \ref{p-thm7.5}). 
\end{rem}

Anyway, by this section, Theorem \ref{p-thm6.1} is now released from the 
deep results of the theory of variations of mixed Hodge structure. 
This means that Theorem \ref{p-thm1.2} is also independent of 
the theory of variations of mixed Hodge structure. 

\section{Finite covers of quasi-abelian varieties}\label{p-sec9}

In this section, we discuss finite covers of abelian and 
quasi-abelian varieties. Let us start with the following well-known 
theorem due to Kawamata--Viehweg. 

\begin{thm}[{see \cite[Main Theorem]{kawamata-viehweg} and 
\cite[Theorem 4]{kawamata-abelian}}]\label{p-thm9.1} 
Let $f\colon X\to A$ be a finite surjective morphism 
from a normal complete variety $X$ to an abelian 
variety $A$. 
Assume that the Kodaira dimension $\kappa (X)$ of $X$ is zero. 
Then $f$ is an \'etale morphism. 
\end{thm}

\begin{proof}
Let $\pi\colon \widetilde X\to X$ be a resolution of singularities from a smooth 
projective variety $\widetilde X$. 
Then $q(\widetilde X)\geq \dim \widetilde X$ since 
$f\circ \pi\colon \widetilde X\to A$ is surjective. 
Therefore, $\widetilde X$ is birationally equivalent to 
an abelian variety by Theorem \ref{p-thm5.1} and 
Corollary \ref{p-cor5.2} since $\kappa (\widetilde X)=0$. 
We consider the following commutative diagram
\begin{equation*}
\xymatrix{
\widetilde X\ar[d]_{\alpha_{\widetilde X}}\ar[r]^\alpha&A\\
\mathcal A_{\widetilde X}\ar[ur]_g&
}
\end{equation*} 
where $\alpha_{\widetilde X}\colon 
\widetilde X\to \mathcal A_{\widetilde X}$ is 
the Albanese map of $\widetilde X$. 
Of course, $\alpha_{\widetilde X}$ is birational and $g$ is a finite 
\'etale morphism between abelian varieties. 
Note that both $X$ and $\mathcal A_{\widetilde X}$ are 
the normalization of $A$ in $\mathbb C(\widetilde X)$. 
Therefore, $X$ is isomorphic to $\mathcal A_{\widetilde X}$ over 
$A$. 
This means that $f\colon X\to A$ is an \'etale morphism. 
\end{proof}

\begin{rem}\label{p-rem9.2}
Kawamata's original 
proof of Theorem \ref{p-thm5.1}, 
which is \cite[Theorem 1]{kawamata-abelian}, 
in \cite{kawamata-abelian} uses Theorem \ref{p-thm9.1} 
(see \cite[Theorem 4]{kawamata-abelian} and 
\cite[Main Theorem]{kawamata-viehweg}). 
However, Ein--Lazarsfeld's approach in \cite[Section 2]{ein-lazarsfeld} 
does not need Theorem \ref{p-thm9.1} (see 
\cite[Theorem 4]{kawamata-abelian} 
and \cite[Main Theorem]{kawamata-viehweg}) 
for the 
proof of Theorem \ref{p-thm5.1} (see 
\cite[Theorem 1]{kawamata-abelian}). 
Therefore, 
there are no problems if 
we use Theorem \ref{p-thm5.1} 
for the proof of Theorem \ref{p-thm9.1}. 
\end{rem}

We can generalize Theorem \ref{p-thm9.1} as follows. 
Theorem \ref{p-thm9.3} is nothing but 
\cite[Theorem 26]{kawamata-abelian}. 

\begin{thm}[Finite covers of 
quasi-abelian varieties]\label{p-thm9.3}
Let $f\colon X\to A$ be a 
finite surjective morphism from 
a normal variety $X$ to 
a quasi-abelian variety $A$. 
Assume that the logarithmic 
Kodaira dimension $\overline \kappa (X)$ of 
$X$ is zero. 
Then 
$f$ is an \'etale morphism. 
\end{thm}

\begin{proof}
Let 
\begin{equation*} 
0\longrightarrow \mathcal G_A\longrightarrow 
A\longrightarrow \mathcal A_A\longrightarrow  0
\end{equation*}  
be the Chevalley decomposition. 
We will prove that $f$ is \'etale by 
induction on $d=\dim \mathcal G_A$. 
If $d=0$, then it is Theorem \ref{p-thm9.1}. 
So, we assume that $d>0$. 
We take a subgroup 
\begin{equation*} 
G_1=\mathbb G_m\times \{1\}\times \cdots \times 
\{1\}\subset \mathbb G^d_m=\mathcal G_A. 
\end{equation*}  
We consider 
\begin{equation*} 
0\longrightarrow G_1\longrightarrow A
\overset {\pi_1}\longrightarrow A_1\longrightarrow 
0. 
\end{equation*}  
Note that $A$ is a principal $G_1$-bundle 
over $A_1$ as an algebraic variety in the Zariski topology 
(see, for example, \cite[Theorems 4.4.1 and 4.4.2]{bcm}). 
As in the proof of Lemma \ref{p-lem3.10}, we can construct 
a $\mathbb P^1$-bundle 
$\overline {\pi_1}\colon 
\overline A\to A_1$ which is a 
partial compactification of 
$\pi_1\colon  A\to A_1$. 
Let $\overline X$ be the normalization 
of $\overline A$ in $\mathbb C(X)$ and 
$\overline f\colon  \overline X\to \overline A$ is the natural map. 
Let $\overline X\to X_1\to A_1$ be the Stein 
factorization of 
$\overline {\pi_1}\circ \overline f\colon  \overline X\to 
A_1$. 
Then we have the following commutative diagram: 
\begin{equation*} 
\xymatrix{
X\ar[r]^{p_1}\ar[d]_{f}&X_1\ar[d]^{f_1}\\
A\ar[r]_{\pi_1}&A_1, 
}
\end{equation*}  
where $f_1\colon  X_1\to A_1$ is a finite morphism from a normal variety $X_1$. 
Since $f_1$ is finite and $A_1$ is a quasi-abelian variety, 
we have $\overline \kappa (X_1)\geq 0$. 
On the other hand, by Theorem \ref{p-thm6.3}, 
we have 
\begin{equation*} 
0=\overline \kappa (X)\geq \overline \kappa (F)+\overline \kappa (X_1), 
\end{equation*}  
where $F$ is a general fiber of $p_1$. 
Note that $\overline \kappa (F)\geq 0$ since $\overline 
\kappa (X)=0$. 
Therefore, we obtain $\overline \kappa (X_1)=\overline \kappa (F)=0$.  
By induction on $d$, $f_1$ is \'etale. 
By replacing $A_1$ (resp.~$A$) with 
$X_1$ (resp.~$A\times _{A_1}X_1$), 
we may assume that 
$f_1$ is the identity. 
\begin{equation*} 
\xymatrix{
X\ar[r]^{p_1}\ar[d]_{f}&A_1\ar@{=}[d]\\
A\ar[r]_{\pi_1}&A_1, 
}
\end{equation*}  
Let $x$ be a general point of $A_1$. 
Then 
\begin{equation*} 
f|_{X_x}\colon  X_x\simeq \mathbb G_m \to A_x\simeq \mathbb G_m
\end{equation*}  
is \'etale. We put $e=\deg f$. 
By construction, there are 
prime divisors $H_1$ and $H_2$ on $\overline A$ 
such that $H_1, H_2\subset \overline A- A$, 
$H_1\sim _{\overline {\pi_1}}H_2$, 
$H_1\ne H_2$, and $H_i$ is a section of 
$\overline {\pi_1}$ for 
$i=1, 2$. 
We can take a nonempty Zariski 
open set $U$ of $A_1$ such that 
\begin{itemize}
\item[(i)] $\overline {p_1}\colon 
\overline X\to A_1$ is smooth 
over $U$. 
\item[(ii)] every fiber of $\overline {p_1}$ is $\mathbb P^1$ over 
$U$. 
\item[(iii)] there are prime 
divisors $D_1$ and $D_2$ on 
$\overline X$ such that $D_1, D_2\subset \overline X -X$, 
$D_1\sim D_2$ over 
$U$, 
$D_1\ne D_2$, and $D_i$ is a section of $\overline {p_1}\colon 
\overline X\to A_1$ over $U$ for $i=1, 2$. 
\item[(iv)] ${\overline f}^*H_i=eD_i$ over $U$ for $i=1, 2$. 
\end{itemize}  
Therefore, we see that $f\colon X\to A$ is 
\begin{equation*} 
\mathbb G_m \times U\to \mathbb G_m \times U
\end{equation*}  
given by 
\begin{equation*} 
(a, b)\mapsto (a^e, b) 
\end{equation*}  
over $U$. 
On the other hand, 
we can construct a 
quasi-abelian variety $A'$ such that 
\begin{equation*} 
\xymatrix{
A'\ar[r]\ar[d]_{h}&A_1\ar@{=}[d]\\
A\ar[r]_{\pi_1}&A_1, 
}
\end{equation*} 
where $h$ is \'etale with 
$\deg h=e$ (see the description of 
quasi-abelian varieties in \ref{p-say4.1}) and that 
$h\colon A'\to A$ is 
\begin{equation*} 
\mathbb G_m \times U\to \mathbb G_m \times U
\end{equation*}  
given by 
\begin{equation*} 
(a, b)\mapsto (a^e, b)  
\end{equation*}  
over $U$, that is, $h$ 
coincides with $f$ over $U$. 
Note that $X$ is normal 
and both $f$ and $h$ are finite. 
Thus $X$ is isomorphic to $A'$ over $A$. 
Hence, we obtain that $f\colon  X\to A$ is \'etale. 
\end{proof}

We will use Theorem \ref{p-thm9.3} 
in the proof of Theorem \ref{p-thm1.2} 
(see the proof of Theorem \ref{p-thm10.1} 
in Section \ref{p-sec10}). 

\section{Quasi-Albanese maps for 
varieties with {$\overline \kappa =0$}}\label{p-sec10}

In this section, 
we give a detailed proof 
of Kawamata's theorem on 
quasi-Albanese maps for varieties 
with $\overline \kappa=0$. 

\begin{thm}[{see \cite[Theorem 28]{kawamata-abelian}}]\label{p-thm10.1}
Let $X$ be a smooth variety such 
that the logarithmic Kodaira 
dimension $\overline \kappa (X)$ 
of $X$ is zero. 
Then the quasi-Albanese map $\alpha\colon X\to A$ 
is dominant and has 
irreducible general fibers. 
\end{thm}

As an easy consequence of 
Theorem \ref{p-thm10.1}, we have: 

\begin{cor}[{see \cite[Corollary 29]{kawamata-abelian} 
and \cite[Theorem I]{iitaka3}}]\label{p-cor10.2}
Let $X$ be a smooth variety such that 
the logarithmic Kodaira dimension $\overline \kappa(X)$ of $X$ is 
zero. 
Then we have $\overline q(X)\leq \dim X$, where 
$\overline q(X)$ is the logarithmic irregularity of $X$. 
Moreover, the equality holds if and only if 
the quasi-Albanese map $\alpha\colon X\to A$ is birational. 
\end{cor}

For the proof of Theorem \ref{p-thm10.1}, we start 
with a useful lemma. 

\begin{lem}\label{p-lem10.3}
Let $A$ be a quasi-abelian variety and let $B$ be a quasi-abelian 
subvariety of $A$. 
Let $B^\dag$ be a variety and let 
$B^{\dag}\to B$ be a finite \'etale cover. 
Then we can construct a finite \'etale cover $A^{\dag}\to A$ such that 
$B^{\dag}$ is a quasi-abelian subvariety of  $A^{\dag}$ satisfying 
\begin{equation*} 
\xymatrix{
B^{\dag}\ar[d]\ar@{^{(}->}[r]&A^{\dag}\ar[d]\\
B\ar@{^{(}->}[r]&A. 
}
\end{equation*} 
\end{lem}
\begin{proof}[Proof of Lemma \ref{p-lem10.3}] 
By Theorem \ref{p-thm4.2}, $B^\dag$ is a 
quasi-abelian variety. 
We consider the Chevalley decompositions: 
\begin{equation*} 
\xymatrix{
0 \ar[r]&\mathcal G_B\ar[r]\ar[d]& B\ar[r]\ar[d]&\mathcal A_B\ar[r]\ar[d]&0 \\
0 \ar[r]&\mathcal G_A\ar[r]& A\ar[r]&\mathcal A_A\ar[r]&0 
}
\end{equation*}  
and 
\begin{equation*} 
\xymatrix{
0 \ar[r]&\mathcal G_{B^\dag}\ar[r]& B^{\dag}\ar[r]&\mathcal A_{B^\dag}\ar[r]&0. 
}
\end{equation*}  
By Poincar\'e reducibility (see, 
for example, \cite{mumford}), we have an \'etale morphism 
\begin{equation*} 
a\colon  \mathcal A_{B^\dag}\times \mathcal A'\to \mathcal A_A 
\end{equation*}  
for some abelian variety $\mathcal A'$. 
By taking the base change 
of $A\to \mathcal A_A$ by $a$, we obtain an \'etale 
cover $A_1\to A$ and 
\begin{equation*} 
\xymatrix{
B^{\dag}\ar[d]\ar[r]&A_1\ar[d]\\
B\ar@{^{(}->}[r]&A. 
}
\end{equation*}  
We note the Chevalley decompositions: 
\begin{equation*} 
\xymatrix{
0 \ar[r]&\mathcal G_{B^\dag}\ar[r]\ar[d]& B^\dag
\ar[r]\ar[d]&\mathcal A_{B^\dag}\ar[r]\ar[d]&0 \\
0 \ar[r]&\mathcal G_{A_1}\ar[r]& A_1\ar[r]&\mathcal A_{B^\dag}\times 
\mathcal A'\ar[r]&0.  
}
\end{equation*}  
By replacing the lattice corresponding to $\mathcal G_{A_1}$ with 
a suitable sublattice, we can construct a finite \'etale morphism 
\begin{equation*} 
A^\dag\to A_1
\end{equation*}  
over $\mathcal A_{B^\dag}\times \mathcal A'$ such that 
\begin{equation*} 
\xymatrix{
B^{\dag}\ar[d]\ar@{^{(}->}[r]&A^{\dag}\ar[d]\\
B\ar@{^{(}->}[r]&A
}
\end{equation*} 
(see the description of quasi-abelian varieties in \ref{p-say4.1}). 
We obtain a desired finite \'etale cover $A^\dag\to A$.  
\end{proof}

Before we prove Theorem \ref{p-thm10.1}, we have to prove 
the following important lemma. 

\begin{lem}[{\cite[Theorem 27]{kawamata-abelian}}]\label{p-lem10.4}
Let $X$ be a normal algebraic variety, let $A$ be a quasi-abelian 
variety, and let $f\colon X\to A$ be a finite morphism. 
Then $\overline \kappa (X)\geq 0$ and 
there are a quasi-abelian subvariety $B$ of $A$, finite 
\'etale 
covers $\widetilde X$ and $\widetilde B$ of $X$ and $B$ 
respectively, and a normal 
algebraic variety $\widetilde Y$ such that: 
\begin{itemize}
\item[(i)] $\widetilde Y$ is finite over $A/B$. 
\item[(ii)] $\widetilde X$ is a principal $\widetilde B$-bundle 
over $\widetilde Y$ as a complex manifold. 
\item[(iii)] $\overline 
\kappa (\widetilde Y)=\dim \widetilde Y=\overline
\kappa (X)$. 
\end{itemize} 
\end{lem}

In \cite{kawamata-abelian}, 
Kawamata claims this statement 
without proof. 
Hence we give a detailed proof for 
the sake of completeness. 
When we treat Iitaka fibrations in the proof of 
Lemma \ref{p-lem10.4}, we have to take care of the 
difference between general fibers and sufficiently general fibers. 

\begin{rem}[Sufficiently general fibers and 
points]\label{p-rem10.5}
Let $f\colon V\to W$ be a morphism 
between varieties. 
Then a {\em{sufficiently general fiber}} (resp.~{\em{general fiber}}) 
$F$ of $f\colon V\to W$ 
means that $F=f^{-1}(w)$, 
where $w$ is any closed point contained in a countable intersection 
of nonempty Zariski open sets (resp.~a nonempty Zariski open set) 
of $W$. 
A sufficiently general fiber is sometimes called a {\em{very 
general fiber}} in the literature. 
A {\em{sufficiently general point}} (resp.~{\em{general point}}) 
$w$ of $W$ is any 
closed point contained in a countable intersection of nonempty 
Zariski open sets (resp.~a nonempty Zariski open set) of $W$. 
\end{rem}

Let us prove Lemma \ref{p-lem10.4}. 

\begin{proof}[{Proof of Lemma \ref{p-lem10.4}}] 
We divide the proof into several steps. 

\setcounter{step}{0}
\begin{step}\label{p-10.4-step1}
Let 
\begin{equation*} 
\xymatrix{
Z\ar[d]_{g}\ar[r]^{\Phi}& Y\\
X\ar@{-->}[ru]&
}
\end{equation*}  
be the logarithmic Iitaka fibration of $X$, that is, 
we take a smooth complete variety $\overline X$ such that 
$D=\overline X-X$ is a simple normal crossing divisor, 
$\overline {X}
\dashrightarrow Y$ is a dominant rational map to a normal projective 
variety $Y$ associated to $|m(K_{\overline X}+D)|$ for 
a sufficiently large and divisible positive integer 
$m$, and $\overline g\colon \overline Z\to \overline X$ is an 
elimination of indeterminacy of 
$\overline X\dashrightarrow Y$ such that 
$Z:=\overline g^{-1}(X)$, $\overline Z- Z$ is a simple 
normal crossing divisor on $\overline Z$, and $g:=\overline g|_Z$.    
Let $y$ be a sufficiently general point of $Y$. 
Then we have $\overline \kappa (Z_y)=0$ by construction, 
where $Z_y:=\Phi^{-1}(y)$. 
Since $f\circ g\colon  Z_y\to f\circ g(Z_y)$ is proper and 
generically finite and 
$f\circ g(Z_y)\subset A$, 
$B_y:=f\circ g(Z_y)$ is a translation of a quasi-abelian 
subvariety of $A$ by Theorem \ref{p-thm4.4}. 
By Theorem \ref{p-thm9.3}, $Z'_y$ is a quasi-abelian 
variety where $Z_y\to Z'_y\to B_y$ is the Stein factorization. 
In particular, the proper birational morphism 
$Z_y\to Z'_y$ is the quasi-Albanese map 
for sufficiently general $y\in Y$. 
Let $y$ be a sufficiently general point of $Y$. 
Note that 
\begin{equation*} 
H_1(Z_y, \mathbb Z)\to 
H_1(A, \mathbb Z)
\end{equation*} 
does not 
depend on $y$ by discreteness. 
Therefore, the image of $Z_y$ by $f\circ g$ in $A$ 
does not depend on $y$ up to translation. 
Hence 
we obtain a quasi-abelian subvariety $B$ of $A$ such that 
$B_y$ is a translation of $B$ for sufficiently general $y\in Y$. 
Without loss of generality, we may 
assume that 
\begin{equation*}
\xymatrix{
Z\ar[r]^-{f\circ g} & A \ar[r]& A/B
}
\end{equation*}
extends 
to 
\begin{equation*}
\xymatrix{
\overline Z\ar[r]^-{\overline {f\circ g}} & \overline A \ar[r]& \overline {A/B}, 
}
\end{equation*}
where $\overline A$ and $\overline {A/B}$ are smooth compactifications 
of $A$ and $A/B$, respectively. 
By applying the rigidity lemma (see, for example, 
\cite[Lemma 4.2.13.~(Rigidity Lemma)]{bs}) to $\overline Z\to Y$ and 
$\overline Z\to \overline {A/B}$, 
we see that 
$B_y=f\circ g(Z_y)\subset A$ is a translation of $B$ for general 
$y\in Y$. For general $y\in Y$, the image of 
\begin{equation*}
H_1(Z_y, \mathbb Z)\to H_1(B_y, \mathbb Z)\simeq 
H_1(B, \mathbb Z)
\end{equation*} 
does not depend on $y$ by discreteness. 
Thus, we have an \'etale cover $\widetilde B$ of $B$ such that 
\begin{equation*}
\xymatrix{
f\circ g\colon Z_y\ar[r]^-{\alpha_{Z_y}}
& \mathcal A_{Z_y}=\widetilde B\ar[r]& B_y, 
} 
\end{equation*}
where $\alpha_{Z_y}\colon Z_y\to \mathcal A_{Z_y}$ is 
the quasi-Albanese map of $Z_y$, for general $y\in Y$ 
and that $\alpha_{Z_y}\colon 
Z_y\to \mathcal A_{Z_y}$ is a proper birational 
morphism for general $y\in Y$. In particular, 
we have $\overline \kappa (Z_y)=0$ for general $y\in Y$. 
\end{step}
\begin{step}\label{p-10.4-step2}
In this step, we will construct $\widetilde X$, 
$\widetilde Y$, and $\widetilde B$ satisfying (i) and (ii). 

By Lemma \ref{p-lem10.3}, 
we take a finite \'etale cover $\widetilde A\to A$ such that 
\begin{equation*} 
\xymatrix{
\widetilde B\ar[d]\ar@{^{(}->}[r]&\widetilde A\ar[d]\\
B\ar@{^{(}->}[r]&A. 
}
\end{equation*}  
By taking base changes, we obtain $\widetilde X$, $\widetilde Z$, and 
the following commutative diagram: 
\begin{equation*} 
\xymatrix{
\widetilde B\ar[d]\ar@{^{(}->}[r]&
\widetilde A\ar[d]&\widetilde X\ar[l]_{\widetilde f}\ar[d]&\widetilde Z
\ar[l]_{\widetilde g}\ar[d]\ar[d]^-p\\
B\ar@{^{(}->}[r]&A&X\ar[l]^{f}&Z. 
\ar[l]^{g}
}
\end{equation*} 
We first assume that $\widetilde X$ is irreducible. 
We can construct a 
logarithmic Iitaka fibration 
$\widetilde \Phi\colon \widetilde Z\to Y'$ as follows. 
Without loss of generality, we may assume that 
there  are a smooth projective variety $\overline Z$ such that 
$\overline Z-Z$ is a simple normal crossing divisor on $\overline Z$ and 
a morphism $a\colon \overline Z\to Y$ such that 
$\Phi=a|_Z\colon  Z\to Y$ in Step \ref{p-10.4-step1}. 
Let $Z^\dag$ be the normalization of 
$\overline Z$ in $\mathbb C(\widetilde Z)$ and let $b\colon 
Z^\dag\to \overline Z$ 
be the natural map. 
Let $H$ be a very ample Cartier divisor on $Y$. 
We consider $\Phi_{|mb^*a^*H|}\colon  Z^\dag\to Y'$ for a sufficiently 
large and divisible positive integer $m$. 
We put $\widetilde \Phi:=\Phi_{|mb^*a^*H|}|_{\widetilde Z}\colon  
\widetilde Z\to Y'$. 
As we saw in Step \ref{p-10.4-step1}, 
there exists a quasi-abelian subvariety $B'$ of $\widetilde A$ 
such that $\widetilde f\circ \widetilde g(\widetilde Z_{y'})$ is a translation 
of $B'$ for general $y'\in Y'$. 
On the other hand, for general $y\in Y$, 
we can take $V\subset \widetilde Z$ such that 
$p\colon V\to Z_y$ is an isomorphism 
since there exists $\alpha_{Z_y}\colon Z_y\to \mathcal A_{Z_y}=\widetilde B$. 
This implies that $B'=\widetilde B$ and 
$\widetilde Z_{y'}\to \widetilde f\circ \widetilde g(\widetilde Z_{y'})$ 
is proper birational for general $y'\in Y'$. 
Therefore, there is a 
rational map $Y'\dashrightarrow \widetilde A/\widetilde B$ such that 
\begin{equation*} 
\xymatrix{
\widetilde Z\ar[r]^{\widetilde \Phi}
\ar[d]_{\widetilde f\circ \widetilde g}&Y'\ar@{-->}[d]\\
\widetilde A\ar[r]&\widetilde A/\widetilde B 
}
\end{equation*}  
(see, for example, \cite[Lemma 14]{kawamata-abelian}). 
Let $\widetilde Y$ be the normalization of 
$\widetilde A/\widetilde B$ in $\mathbb C(Y')$. 
We put $X'=\widetilde A\times _{\widetilde A/\widetilde B}\widetilde Y$. 
Then $X'$ is normal and is birationally equivalent to 
$\widetilde X$. 
We note that $\widetilde X$ and $X'$ are both finite over $\widetilde A$. 
Thus $X'$ is isomorphic to $\widetilde X$ over $\widetilde A$. 
We also note that $\widetilde Y$ is 
finite over $A/B$ since $\widetilde A/\widetilde B$ 
is finite over $A/B$. 
By construction, $\widetilde X$ 
is a principal $\widetilde B$-bundle over $\widetilde Y$. 
When $\widetilde X$ is reducible, 
we replace $\widetilde X$ with a suitable 
irreducible component of $\widetilde X$. 
Then the above argument works. 
Anyway, we can construct $\widetilde X$, $\widetilde Y$, 
and $\widetilde B$ satisfying (i) and (ii). 
\end{step}

\begin{step}\label{p-10.4-step3} 
All we have to show is $\overline 
\kappa (\widetilde Y)=\dim \widetilde Y=\overline 
\kappa (X)$. 
Since $\widetilde Y$ is finite over $\widetilde A/\widetilde B$, we have 
$\overline \kappa (
\widetilde Y)\geq 0$. 
We assume that $\overline \kappa (Y)<\dim \widetilde Y$. 
By applying the results obtained in Steps \ref{p-10.4-step1} 
and \ref{p-10.4-step2} to 
$\widetilde Y\to A/B$, 
we obtain an \'etale cover $\widetilde Y'$ with the following commutative 
diagram: 
\begin{equation*} 
\xymatrix{
\widetilde X'=\widetilde X\times _{\widetilde Y}\widetilde Y'\ar[d]\ar[r]
& \widetilde Y'\ar[d]\ar[r] & 
W\ar[dd]\\
\widetilde X \ar[d]\ar[r]&\widetilde Y \ar[d]&\\
A \ar[r]&A/B\ar[r]&A/C,  
} 
\end{equation*}  
where $C$ is a quasi-abelian subvariety of $A$ such that 
$B\subset C$. 
Note that $W$ is finite over $A/C$ and that $\dim W=\overline 
\kappa (Y)$. 
We can easily see that every fiber of $\widetilde X'\to W$ is 
a quasi-abelian variety and 
\begin{equation*} 
\overline \kappa (\widetilde X')=\overline \kappa (\widetilde X)=\overline 
\kappa (X). 
\end{equation*}  
By the easy addition formula, 
we obtain 
\begin{equation*} 
\overline \kappa (\widetilde X')\leq \dim W<\dim Y=\overline \kappa (X). 
\end{equation*}  
This is a contradiction. 
Therefore, we have $\dim \widetilde Y=\overline \kappa (\widetilde Y)$. 
\end{step}
We have desired $\widetilde X$, $\widetilde B$, and $\widetilde Y$. 
We finish the proof of Lemma \ref{p-lem10.4}. 
\end{proof}

Let us start the proof of Theorem \ref{p-thm10.1}. 

\begin{proof}[Proof of Theorem \ref{p-thm10.1}]
By using the Stein factorization, we obtain 
\begin{equation*} 
\alpha\colon  X\overset{q}\longrightarrow Z\overset{p}\longrightarrow A
\end{equation*} 
where $q$ is dominant, $q$ has irreducible 
general fibers, $p$ is finite, and $Z$ is normal. 
It is sufficient to prove that $p$ is an isomorphism. 
We assume that $\overline \kappa (Z)>0$. 
Then, by Lemma \ref{p-lem10.4}, 
we obtain an \'etale cover $\widetilde Z\to Z$ such that 
$\widetilde Z\to W$ is a principal $G$-bundle 
for some quasi-abelian variety $G$ with $\overline 
\kappa (W)=\dim W=\overline \kappa (Z)>0$. 
We consider $r\colon 
\widetilde X=X\times _Z\widetilde Z\to \widetilde Z\to W$. 
Since $\overline \kappa (\widetilde X)=\overline \kappa 
(X)=0$, 
$\overline \kappa (F)\geq 0$ for a sufficiently 
general fiber $F$ of $r$. 
By Theorem \ref{p-thm6.1}, 
we obtain 
\begin{equation*} 
0=\overline \kappa (X)=\overline \kappa (\widetilde X)\geq \overline \kappa (W)
+\overline \kappa (F)\geq \overline \kappa (Z)>0. 
\end{equation*}  
This is a contradiction. 
Therefore, we obtain $\overline \kappa (Z)=0$. 
By Theorem \ref{p-thm4.4}, 
$\overline \kappa(p(Z))=0$ and $p(Z)$ is a quasi-abelian 
variety. 
By Theorem \ref{p-thm9.3}, 
we obtain that $p\colon Z\to p(Z)$ is \'etale. 
In particular, $Z$ is a quasi-abelian variety (see Theorem \ref{p-thm4.2}). 
This means that $p$ is an isomorphism since $\alpha\colon X\to A$ is 
a quasi-Albanese map of $X$. 
Thus, we obtain that 
$\alpha\colon X\to A$ is dominant and has irreducible 
general fibers. 
\end{proof}

We close this section with the proof of Corollary \ref{p-cor10.2}. 

\begin{proof}[Proof of Corollary \ref{p-cor10.2}]
Let $\alpha\colon X\to A$ be a quasi-Albanese map. 
By Theorem \ref{p-thm10.1}, $\alpha$ is dominant. 
Note that $\dim A=\overline q(X)$. 
Therefore, we have $\overline q(X)\leq \dim X$. 
By Theorem \ref{p-thm10.1}, the general fibers of $\alpha$ are irreducible. 
Thus, $\alpha$ is birational if and only if $\dim X=\dim A=\overline q(X)$.  
\end{proof}

\section{Proof of Theorem \ref{p-thm1.3} and 
Corollaries \ref{p-cor1.4} and \ref{p-cor1.5}}\label{p-sec11}

In this final section, we give a proof of Theorem 
\ref{p-thm1.3} following \cite{fmpt}. Then we prove 
Corollaries \ref{p-cor1.4} and \ref{p-cor1.5} as easy applications. 

\begin{proof}[Proof of Theorem \ref{p-thm1.3}]
Let $\alpha\colon X\to A$ be the quasi-Albanese 
map (see Theorem \ref{p-thm1.1}). By Corollary \ref{p-cor10.2}, 
we see that $\alpha$ is birational. 
\setcounter{step}{0}
\begin{step}\label{p-1.3-step1}
Let 
\begin{equation*}
\xymatrix{
0\ar[r]& \mathbb G^d_m\ar[r] &A\ar[r]^-\pi&  B\ar[r]& 0 
}
\end{equation*}
be the Chevalley decomposition as in Definition \ref{p-def2.8}. 
Then $A$ is a principal $\mathbb G^d_m$-bundle 
over an abelian variety $B$ in the Zariski topology 
(see, for example, 
\cite[Theorems 4.4.1 and 4.4.2]{bcm})
and there is a natural 
completion $\overline \pi\colon \overline A\to B$ of $\pi\colon A\to B$
where $\overline A$ is a 
${\mathbb P}^d$-bundle over $B$ (see the proof of 
Lemma \ref{p-lem3.8}). We set 
$\Delta_{\overline A}:=\overline A -A$. 
Then $\Delta_{\overline A}$ is a simple 
normal crossing divisor on 
$\overline A$. In particular,  
$(\overline A, \Delta_{\overline A})$ is a log canonical pair. 
Let $\overline \alpha\colon \overline X\to 
\overline A$ be a compactification of 
$\alpha \colon X\to A$, that is, 
$\overline X$ is a smooth complete 
algebraic variety containing $X$, 
$\Delta_{\overline X}:=
\overline X-X$ is a simple 
normal crossing divisor 
on $\overline X$, and $\overline\alpha$ is a morphism extending $\alpha$.
\begin{claim}\label{p-1.3-claim}
Let $D$ be an irreducible component of 
$\Delta_{\overline X}$ such that 
$\overline \alpha(D)$ is a divisor. 
Then $\overline \pi\colon \overline \alpha (D)\to B$ 
is dominant.  
\end{claim}
\begin{proof}[Proof of Claim] 
We set $D_1:=\overline \alpha (D)$. 
If $\overline \pi\colon D_1\to B$ is not dominant, 
then we can write $D_1={\overline \pi}^*D_2$ 
for some prime divisor $D_2$ on $B$. 
Since every log canonical center 
of $(\overline A, \Delta_{\overline A})$ 
dominates $B$, $D_1$ does not 
contain any log canonical centers. Hence, in particular, 
it is not a component of $\Delta_{\overline A}$. 
Thus we have 
\begin{equation*}
\mathrm{mult}_D\left(K_{\overline X}+\Delta_{\overline X} 
-\overline \alpha^*(K_{\overline A}+\Delta_{\overline A})\right)=1.  
\end{equation*}
Let $E$ be any $\overline \alpha$-exceptional divisor on $\overline 
X$ such that $\overline \alpha (E)$ is not a log 
canonical center of $(\overline A, \Delta_{\overline A})$. 
Then 
\begin{equation*}
\mathrm{mult}_E\left(K_{\overline X}+\Delta_{\overline X}-\overline 
\alpha^*(K_{\overline A}+\Delta_{\overline A})\right)\geq 1
\end{equation*}
holds since $K_{\overline A}+\Delta_{\overline A}$ is Cartier and $\Supp 
\overline \alpha^*\Delta_{\overline A}\subset 
\Supp \Delta_{\overline X}$. 
Therefore,  
\begin{equation*}
K_{\overline X}+\Delta_{\overline X} 
-\overline \alpha^*(K_{\overline A}+\Delta_{\overline A}) 
\geq \varepsilon \overline \alpha^*D_1 
\end{equation*} 
holds for some $0<\varepsilon \ll 1$ 
since the support of 
$D_1$ does not contain any log canonical centers of $(\overline A, 
\Delta_{\overline A})$. 
By construction, we have $K_{\overline A}+\Delta_{\overline A}
\sim 0$ (see the proof of Lemma \ref{p-lem3.8}). 
Hence we obtain  
\begin{equation*}
0=\overline \kappa (X)=\kappa(\overline X, K_{\overline X}
+\Delta_{\overline X})\geq \kappa (\overline X, 
{\overline \alpha}^*D_1)=\kappa (\overline A, D_1)
=\kappa (B, D_2)>0, 
\end{equation*} 
where the last inequality follows from the fact that $D_2$ 
is a nonzero effective divisor on the abelian variety $B$.
This  contradiction proves the claim.
\end{proof}
\end{step}

\begin{step}\label{p-1.3-step2}
We assume that there exists an irreducible component 
$D$ of $\Delta_{\overline X}$ such that 
$\overline \alpha (D)$ is a divisor 
with $\overline \alpha(D)\not 
\subset \overline A-A$. 
We set 
$D':=\overline \alpha (D)\cap A$. 
By Claim in Step \ref{p-1.3-step1}, $D'$ dominates $B$. 
Therefore,  
we can find a subgroup $\mathbb G_m$ of $A$ 
such that $\varphi|_{D'}\colon D'\to A_1$ is dominant, where 
\begin{equation*}
\xymatrix{
0\ar[r]& \mathbb G_m\ar[r] &A\ar[r]^-\varphi&  A_1\ar[r]& 0. 
} 
\end{equation*}  
Note that $A$ is a principal $\mathbb G_m$-bundle over $A_1$ in 
the Zariski topology (see, for example, 
\cite[Theorems 4.4.1 and 4.4.2]{bcm}). 
We take a compactification 
\begin{equation*}
\xymatrix{
f^\dag
\colon X^\dag \ar[r]^-{\alpha^\dag}& 
A^\dag \ar[r]^-{\varphi^\dag}& A^\dag_1
}
\end{equation*}
of 
\begin{equation*}
\xymatrix{
f
\colon X \ar[r]^-{\alpha}& A\ar[r]^-{\varphi}& A_1, 
}
\end{equation*}
where $X^\dag$, $A^\dag$, and $A^\dag_1$ are smooth complete 
algebraic varieties such that 
$X^\dag- X$, $A^\dag- A$, 
and $A^\dag_1- A_1$ are 
simple normal crossing divisors. 
The general fiber of $f^\dag$ is obviously $\mathbb P^1$ by construction. 
Let $F$ be a general fiber of $f$. 
Since $\varphi|_{D'}\colon D'\to A_1$ is 
dominant, we have $\# \left(\mathbb P^1\setminus F\right)\geq 3$. 
This implies $\overline \kappa (F)=1$. 
Note that $A_1$ is a quasi-abelian variety. 
Hence we have $\overline \kappa (A_1)=0$. 
By Theorem \ref{p-thm6.3}, 
we obtain  
\begin{equation*}
0=\overline \kappa (X)\geq \overline \kappa (F)
+\overline \kappa (A_1)=1. 
\end{equation*} 
This is a contradiction. Thus, 
every irreducible component of $\Delta_{\overline X}$ 
which is not contracted by $\overline \alpha$ is mapped to 
$\Delta_{\overline A}$. 
\end{step}
\begin{step}\label{p-1.3-step3}
Let $\Delta'$ be the union of 
$\Delta_{\overline X}$ and the exceptional locus 
$\Exc(\overline \alpha)$ of $\overline 
\alpha\colon \overline X\to \overline A$. 
Note that $\Exc(\overline \alpha)$ 
is of pure codimension one by \cite[Chapter 2, Section 4.4, 
Theorem 2.16]{shafarevich}. 
Therefore, $\Delta'$ is a divisor on $\overline X$ with 
$\Delta'\geq \Delta_{\overline X}$. 
We put $Z:=\overline \alpha(\Delta')\cap A$. 
By Step \ref{p-1.3-step2}, 
we see that $Z$ is a closed subset of $A$ with 
$\codim_A Z\geq 2$. By definition, 
$\overline \alpha\colon \overline X\to \overline A$ 
is an isomorphism over $A\setminus Z$. 
Over $A$, we can easily check that 
$\overline \alpha^{-1}(Z)=\Exc(\overline \alpha)=\Delta'$ 
holds by 
\cite[Chapter 2, Section 4.4, 
Theorem 2.16]{shafarevich}. 
Hence, $\alpha\colon X\setminus \alpha^{-1}(Z)\to A\setminus Z$ 
is an isomorphism and $\alpha^{-1}(Z)$ is 
of pure codimension one. 
This is what we wanted. 
\end{step}
We finish the proof of Theorem \ref{p-thm1.3}.
\end{proof}

\begin{rem}\label{p-rem11.1}
In the proof of Theorem \ref{p-thm1.3}, 
$\overline A$ never coincides with $A^\dag$ when 
$A_1\ne B$. If $A=\mathbb G^2_m$, then 
$\overline A=\mathbb P^2$ and $A^\dag_1=\mathbb P^1$. 
\end{rem}

Corollary \ref{p-cor1.4} easily follows from 
Theorem \ref{p-thm1.3}. 

\begin{proof}[Proof of Corollary \ref{p-cor1.4}] 
If $X\simeq \mathbb G^n_m$, then 
$\overline \kappa (X)=0$ and $\overline q(X)=n$ hold 
by Theorem \ref{p-thm4.3}. 
Hence it is sufficient to prove that $X\simeq 
\mathbb G^n_m$ holds under the assumption 
that $\overline \kappa (X)=0$ and $\overline q(X)=n$. 
Let $\alpha \colon X\to A$ be the quasi-Albanese map. 
By Theorem \ref{p-thm1.3}, we can take 
a closed subset $Z$ of $A$ such that 
$\codim_A Z\geq 2$, $\alpha^{-1}(Z)$ is 
of pure codimension one, and $\alpha\colon X\setminus 
\alpha^{-1}(Z)\to A\setminus Z$ is an isomorphism. 
We note that $X\setminus \alpha^{-1}(Z)$ is affine 
since $X$ is a smooth affine variety and 
$\alpha^{-1}(Z)$ is of pure codimension one. 
This implies that $A\setminus Z$ is also affine. 
Thus we obtain $Z=\emptyset$ since $\codim_A Z\geq 2$ 
(see, for example, \cite[Lemma 6]{iitaka2}). 
Hence $\alpha\colon X\to A$ is an isomorphism. 
In particular, $X\simeq A\simeq \mathbb G^n_m$. 
Note that $A$ is quasi-abelian and affine. 
This is what we wanted. 
\end{proof}

Although Corollary \ref{p-cor1.5} is almost obvious, 
we prove it for the sake of completeness. 
 
\begin{proof}[Proof of Corollary \ref{p-cor1.5}] 
If $\codim _A(A\setminus X)\geq 2$, then we have $\overline 
\kappa (X)=0$ by Lemma \ref{p-lem2.5}. 
We assume that $\overline \kappa (X)=0$ holds. 
Then $X\hookrightarrow A$ is nothing but the quasi-Albanese 
map and $\codim _A(A\setminus X)\geq 2$ by 
Theorem \ref{p-thm1.3}. 
We finish the proof of Corollary \ref{p-cor1.5}. 
\end{proof}

\end{document}